\documentclass[12pt]{amsart}

\usepackage{amsmath,amsthm,amscd,amsfonts,amssymb,epic,eepic,bbm,xcolor}
\usepackage[pagebackref,colorlinks=true,linkcolor=blue,citecolor=blue]{hyperref}

\allowdisplaybreaks
\newtheorem{thm}{Theorem}[section]
\newtheorem{cor}[thm]{Corollary}

\newtheorem{lem}[thm]{Lemma}

\newtheorem{prop}[thm]{Proposition}
\theoremstyle{remark}
\newtheorem*{rem}{Remark}

\newcounter{remarkscounter}

\numberwithin{equation}{section}
\newcommand{\A}{\mathbb{A}}
\newcommand{\GL}{\mathrm{GL}}
\newcommand{\gl}{\mathfrak{gl}}
\newcommand{\SL}{\mathrm{SL}}
\newcommand{\ZZ}{\mathbb{Z}}

\newcommand{\QQ}{\mathbb{Q}}

\newcommand{\lto}{\longrightarrow}

\newcommand{\OO}{\mathcal{O}}
\newcommand{\CC}{\mathbb{C}}
\newcommand{\RR}{\mathbb{R}}
\newcommand{\GG}{\mathbb{G}}

\newcommand{\Sp}{\mathrm{Sp}}

\newcommand{\quash}[1]{}

\theoremstyle{definition}

\renewcommand{\bar}{\overline}

\numberwithin{equation}{subsection}

\newcommand{\one}{\mathbbm{1}}

\begin{document}

\title[A Summation formula for triples of quadratic spaces]{A summation formula for triples of quadratic spaces}

\author{Jayce R. Getz}
\address{Department of Mathematics\\
Duke University\\
Durham, NC 27708}
\email{jgetz@math.duke.edu}

\author{Baiying Liu}
\address{Department of Mathematics\\
Purdue University\\
West Lafayette, IN 47907}
\email{liu2053@purdue.edu}


\thanks{The first named author is thankful for partial support provided by NSF grant DMS 1405708. The second named author is partially supported by NSF grant DMS 1702218 and by a start-up fund from the Department of Mathematics at Purdue University.
A portion of this work was completed during the first author's stay at the Institute for Advanced Study, and he thanks the Charles Simonyi endowment for their support.
Any opinions, findings, and conclusions or recommendations expressed in this material are those of the authors and do not necessarily reflect the views of the National Science Foundation.}

\begin{abstract}
Let $V_1,V_2,V_3$ be a triple of even dimensional vector spaces over a number field $F$ equipped with nondegenerate quadratic forms $\mathcal{Q}_1,\mathcal{Q}_2,\mathcal{Q}_3$, respectively.  Let
\begin{align*}
Y  \subset \prod_{i=1}V_i
\end{align*}
be the closed subscheme consisting of $(v_1,v_2,v_3)$ on which $\mathcal{Q}_1(v_1)=\mathcal{Q}_2(v_2)=\mathcal{Q}_3(v_3)$.  Motivated by conjectures of Braverman and Kazhdan and related work of Lafforgue, Ng\^o, and Sakellaridis we prove an analogue of the Poisson summation formula for certain functions on this space.  
\end{abstract}

\maketitle

\tableofcontents

\section{Introduction} \label{sec:intro}

Godement and Jacquet \cite{GodementJacquetBook}, generalizing Tate's thesis, proved the functional equation and analytic continuation of the standard $L$-function of an automorphic representation of $\GL_n$
as a consequence of the Poisson summation formula on $\mathfrak{gl}_n$.  This formula states that for $F$ a number field with ring of adeles $\A_F$, $\psi:F \backslash \A_F \to \CC^\times$ a nontrivial additive character, $f \in \mathcal{S}(\mathfrak{gl}_n(\A_F))$ and $g \in \GL_n(\A_F)$ one has
\begin{align} \label{glnPS}
\sum_{\gamma \in \mathfrak{gl}_n(F)}f(\gamma g) =\sum_{\gamma \in \mathfrak{gl}_n(F)}\lvert \det g \rvert^{-n}\widehat{f}(g^{-1}\gamma)\,,
\end{align}
where $\widehat{f}(X):=\int_{\gl_n(\A_F)}f(Y)\psi(\mathrm{tr}\, XY)dY$ is the Fourier transform of $f$.   

Braverman and Kazhdan \cite{BK-lifting} have suggested that this is but the first case of a general phenomenon.  Let $G$ be a connected split reductive group over $F$.
  For each representation 
$$
r:{}^LG^{\circ} \lto \GL_n
$$
of the neutral component of the $L$-group ${}^LG$ of $G$ satisfying certain assumptions, they conjectured the existence of a corresponding Fourier transform and a Poisson summation formula.  The summation formula should imply the functional equation and meromorphic continuation of the Langlands $L$-function $L(s,\pi,r)$  attached to $r$ and a cuspidal automorphic representation $\pi$ of $G(\A_F)$.  There has been a great deal
of interest in the conjectures of Braverman and Kazhdan and related approaches recently, and we mention in particular the work in
\cite{BouthierNgoSakellaridis,Cheng:Ngo:BK,Getz:RSMonoid,GetzBK,
LafforgueJJM,WWLiSat,WWLi:Zeta,WWLi:Towards,SakellaridisSat,SakellaridisSph,Shahidi:LF,Shahidi:FT}.  Ng\^o has emphasized the relationship between the approach of Braverman and Kazhdan and Langlands' beyond endoscopy proposal \cite{LanglandsBE}, as well as the relationship between Braverman and Kazhdan's work and Vinberg's theory of reductive monoids \cite{Vinberg:reductive:semigroups}.  The basic observation here linking Godement and Jacquet's theory and the theory of monoids is that $\mathfrak{gl}_n$ is a monoid with unit group $\GL_n$.

However, to establish the functional equation and meromorphic continuation of $L$-functions, the monoidal structure, though convenient, is not strictly necessary.  If one is studying $L$-functions of cuspidal automorphic representations of $G(\A_F)$, the bare minimum one needs is a $G$-scheme with a Zariski-open orbit and a summation formula like \eqref{glnPS} for the $G$-scheme.  This is what is really used in \cite[\S 12]{GodementJacquetBook}, and there are other examples in which the $G$-variety in question is spherical (but not necessarily a reductive monoid) that are investigated in \cite{SakellaridisSph}.  We note that Garrett's integral representation of the triple product $L$-function, which plays a key role in this paper, is discussed in \S 4.5 of loc.~cit.  

In the present paper we focus on proving summation formulae for schemes 
admitting natural actions of reductive groups with Zariski open orbits, generalizing the standard representation of $\GL_n$ in the Godement-Jacquet case. As pointed out to the authors by Y.~Sakellaridis, these summation formulae are the first of their kind, in the sense that this is the first case where such a summation formula has been proven when the underlying scheme is not a flag manifold (the case of flag manifolds is treated in \cite{BK:normalized}).  

Let $d_1,d_2,d_3$ be three positive even integers, let $V_i=\GG_a^{d_i}$, $V:=\oplus_{i=1}^3V_i$.  Then $V(F)$ is an $F$-vector space.  For each $i$ let $\mathcal{Q}_i$ be a nondegenerate quadratic form on $V_i(F)$.
Let
 $Y \subset V$ be the subscheme whose points in an $F$-algebra $R$ are given by 
\begin{align*}
Y(R):&=\{ (y_1,y_2,y_3) \in V(R):\mathcal{Q}_1(y_1)=\mathcal{Q}_2(y_2)=\mathcal{Q}_3(y_3)\}\,.
\end{align*}
Let $J_{i}$ be the matrix of $\mathcal{Q}_i$ (see \eqref{Ji}) and let
\begin{align} \label{H}
H(R):=\left\{ (g_1,g_2,g_3) \in \prod_{i=1}^3\GL_{d_i}(R): g_i J_{i}^{-1}({}^tg_i)J_{i}=\lambda I_{d_i} \textrm{ for some } \lambda \in R^\times \right\}.
\end{align}
This is a subgroup of the product of the orthogonal similitude groups attached to the $\mathcal{Q}_i$.  It comes equipped with a character 
\begin{align} \label{lambda}
\lambda:H \lto \GG_m\,,
\end{align}
whose value on $(g_1,g_2,g_3)$ is the similitude norm of $g_1$ (which is equal to the similitude norms of $g_2$ and $g_3$ by definition).  
It is easy to see that the natural action of $H$ on $V$ preserves $Y$.  Using Witt's theorem it is also easy to see that the action of $H$ on $Y$ has a Zariski-open  orbit $Y^{\mathrm{ani}}$, namely the orbit of all vectors $(v_1,v_2,v_3)$ such that $\mathcal{Q}_i(v_i) \neq 0$.  We let $Y^{\mathrm{sm}} \subset Y$ be the smooth locus, it is precisely the subscheme of triples $(y_1,y_2,y_3)$ such that no two $y_i$ are zero; thus we have a triple of schemes
$$
Y^{\mathrm{ani}} \subset Y^{\mathrm{sm}} \subset Y\,,
$$
all preserved by the action of $H$.

 Our goal in this paper is to formulate and prove a Poisson summation formula for $Y(F)$.  Let $\mathrm{Sp}_6$ be the symplectic group (on a $6$-dimensional vector space) and let $K \leq \mathrm{Sp}_6(\A_F)$ be a maximal compact subgroup such that $K^\infty$ is $\mathrm{Sp}_6(\A_F^\infty)$-conjugate to $\mathrm{Sp}_6(\widehat{\OO}_F)$.
Let $\psi:F \backslash \A_F \to \CC^\times$ be a nontrivial character.  Let $P \leq \mathrm{Sp}_6$ be the standard Siegel parabolic subgroup (see \eqref{Siegel}) and let $X:=[P,P] \backslash \mathrm{Sp}_6$.  Using an idea of
Braverman and Kazhdan, we defined a Schwartz space in \cite{Getz:Liu:BK}:
$$
\mathcal{S}_{BK}(X(\A_F),K)\,.
$$
It is a subspace of the space of smooth functions on $X(\A_F)$ that are $K$-finite.  
We also defined a  Fourier transform
\begin{align*}
\mathcal{F}:=\mathcal{F}_{BK,\psi}:\mathcal{S}_{BK}(X(\A_F),K) \lto \mathcal{S}_{BK}(X(\A_F),K)
\end{align*}
(see \S \ref{ssec:Schwartz:spaces}).  

Let $G:=\mathrm{SL}_2^3$.  
There is a natural embedding $\mathrm{SL}_2^{3} \to \mathrm{Sp}_6$ and we sometimes identify  $G$ with its image (see \eqref{Gembed}).
The quotient $X(F) /G(F)$ is a finite set.  The unique Zariski open orbit admits a representative $\gamma_0$ such that the stabilizer in $G$ of $\gamma_0$ is the unipotent group whose points in an $F$-algebra $R$ are given by 
\begin{align} \label{N0}
N_0(R):&=\left\{\Big(\left(\begin{smallmatrix} 1 & t_1 \\ & 1 \end{smallmatrix} \right),\left(\begin{smallmatrix} 1 & t_2\\ & 1 \end{smallmatrix} \right),\left(\begin{smallmatrix} 1 & t_3\\  & 1 \end{smallmatrix} \right) \Big): t_i \in R, \sum_{i=1}^3t_i=0\right\}
\end{align}
(see \S \ref{ssec:BK}).
Let $\mathcal{S}(V(\A_F))$ be the usual Schwartz space. For 
$$
(f_1,f_2) \in \mathcal{S}_{BK}(X(\A_F),K) \times \mathcal{S}(V(\A_F)) \quad \textrm{ and } \quad y \in Y^{\mathrm{sm}}(\A_F)\,,
$$
 define
\begin{align}
I(f_1,f_2)(y)=\int_{N_{0}(\A_F) \backslash G(\A_F)} f_1(\gamma_{0}g)\rho(g)f_2(y)dg\,.
\end{align}
Here $\rho:=\rho_\psi$ is the Weil representation (see \S \ref{ssec:Weil}).  The appearance of the Weil representation is the reason we have assumed that the dimensions of the $V_i$ are even; if some of them were odd then we would have to work with a product of symplectic and metaplectic groups instead of $G$.

Our summation formula is as follows:

\begin{thm} \label{thm:intro}
For $(f_1,f_2) \in \mathcal{S}_{BK}(X(\A_F),K) \times \mathcal{S}(V(\A_F))$ such that $f_1$ and $\mathcal{F}(f_1)$ satisfy  \eqref{A:vanish} and $f_2$ satisfies 
\eqref{vanish}, one has
\begin{align*}
\sum_{\gamma \in Y^{\mathrm{sm}}(F)}I(f_1,f_2)(\gamma)=\sum_{\gamma \in Y^{\mathrm{sm}}(F)}I(\mathcal{F}(f_1),f_2)(\gamma)\,.
\end{align*}
\end{thm}
\noindent We also have the following corollary, proved below in Corollary \ref{cor:main}:
\begin{cor} 
\label{cor:intro}  Let $h \in H(\A_F)$.  For 
$$
(f_1,f_2) \in \mathcal{S}_{BK}(X(\A_F),K) \times \mathcal{S}(V(\A_F))
$$ 
such that $f_1$, $\mathcal{F}(f_1)$ satisfy \eqref{A:BK}, \eqref{A:vanish} and  $f_2$  satisfies \eqref{vanish}, one has
\begin{align*}
\sum_{\xi \in Y^{\mathrm{sm}}(F)} I(f_1,f_2)(h^{-1}\xi)=\sum_{\xi \in Y^{\mathrm{sm}}(F)} |\lambda(h)|^{\sum_{i=1}^3d_i/2-2} I(\mathcal{F}(f_1),f_2)(\lambda(h) h^{-1}\xi)\,.
\end{align*} 
\end{cor}

We now outline the proof of Theorem \ref{thm:intro}.
In \cite{Getz:Liu:BK} following an argument of Braverman and Kazhdan we proved a summation formula of the form
\begin{align} \label{basic:id}
\sum_{\gamma \in X(F)} f_1(\gamma g) =\sum_{\gamma \in X(F)} \mathcal{F}(f_1)(\gamma g)+\textrm{boundary terms}\,,
\end{align}
where $f_1 \in \mathcal{S}_{BK}(X(\A_F),K)$.  Given $f_2 \in \mathcal{S}(V(\A_F))$ one can form a product of three theta functions 
$$
\Theta_{f_2}(g):=\sum_{\gamma \in V(F)}\rho(g)f_2(\gamma)\,,
$$
in the usual manner (see \S \ref{ssec:Theta}).  We view $\Theta_{f_2}$ as an automorphic form on $G(\A_F)=\mathrm{SL}_2^3(\A_F)$.  One takes this automorphic form, integrates it against the identity \eqref{basic:id}, and then unfolds.  The resulting sum is indexed by the finite set $X(F)/G(F)$.  The summand corresponding to the Zariski-open orbit involves an integral over $N_0(F) \backslash N_0(\A_F)$, where $N_0$ is defined as in \eqref{N0}.  This integral eliminates the contribution of all $\gamma \in V(F)$ that are not in $Y(F)$.  Using this one obtains Theorem \ref{thm:intro}.  Since \eqref{basic:id} is essentially equivalent to the functional equations of certain degenerate Siegel Eisenstein series, another way of viewing this proof is that we are substituting $\Theta_{f_2}$ into Garrett's integral representation of the triple product $L$-function \cite{GarrettTripleAnnals,PSRallisTriple}.  {We note that we do not need the full strength of the summation formula proven in \cite{Getz:Liu:BK}.  The version we use is given in Corollary \ref{cor:BK}.  It is slightly more general than that of \cite{BK:normalized}.  However, we still must make use of the growth estimates on elements of the Schwartz space obtained in \cite{Getz:Liu:BK}.  These bounds are not proven in \cite{BK:normalized}.}

This procedure for producing new summation formulae from old is novel and deserves to be studied carefully with a view to generalizations.  The formal argument is short (see \S \ref{sec:summation}).  However, it takes substantial space to make it rigorous by proving various bounds and computing various integrals for unramified data. 

We close the introduction by outlining the sections of the paper.  In \S \ref{sec:groups:orbits} we introduce the groups and homogeneous spaces relevant for the unfolding procedure mentioned above.  We also record  representatives for $X(F)/G(F)$ and the stabilizers of these elements. In \S \ref{ssec:Plucker} we use the Pl\"ucker embedding of $X$ to give a notion of the size for an element of $X(F_v)$ for places $v$ of $F$.  In \S \ref{ssec:Weil} we recall and set notation for the Weil representation.  
 
We define local integrals 
 attached to the open orbit in $X(F)/G(F)$ in \S \ref{sec:Schwartz}.  The full version of Theorem \ref{thm:intro} is stated as Theorem \ref{thm:main}.  In \S \ref{sec:summation} we prove this theorem
modulo proving the absolute convergence of several sums.  The remainder of the paper {(with the exception of \S \ref{sec:appendix})} is devoted to proving these absolute convergence statements.  In each case, the absolute convergence statements amount to bounding local integrals and then bounding their sum over $F$-points of certain schemes.  The local integrals are computed in the unramified case in \S \ref{sec:unr}.  In \S \ref{sec:bound:na} we bound the non-Archimedean local integrals when the data are ramified.  The Archimedean case is treated in \S \ref{sec:bound:arch}.  In each case the arguments are straightforward.  The key point is to use  the bounds on functions in $\mathcal{S}_{BK}(X(\A_F),K)$ established by the authors in \cite{Getz:Liu:BK}; these bounds are given in terms of the Pl\"ucker embedding of $X(F)$.  In \S \ref{sec:AC} we use the bounds established in \S \ref{sec:unr}, \S \ref{sec:bound:na} and \S \ref{sec:bound:arch} to prove the absolute convergence statements used in \S \ref{sec:summation}.  
{
In \S \ref{sec:appendix} we prove a vanishing result necessary for the proof of our main theorem.
}

\section*{Acknowledgments}
The authors would like to thank S.~Kudla, W-W.~Li, B.C.~Ng\^o, Y.~Sakellaridis, F.~Shahidi,
and Z.~Yun for useful conversations and comments.   {The anonymous referees also deserve thanks for their careful reading of the paper, for pointing out several typos, and for suggesting we add a list of symbols.}
The authors also would like to thank H.~Hahn for help with editing and her constant encouragement.

\section{Groups and orbits}
\label{sec:groups:orbits} 
 
For this section we let $F$ be a field of characteristic zero. 
For each $i$, let
\begin{align*}
\langle \,,\,\rangle:V_i(F) \times V_i(F) &\lto F\\
(x,y) &\longmapsto {}^txy
\end{align*}
be the ``standard'' inner product and let
\begin{align*}
\langle \,,\,\rangle_i:V_i(F) \times V_i(F) \lto F
\end{align*}
be the (nondegenerate) inner product corresponding to $\mathcal{Q}_i$:
\begin{equation}\label{Qi}
    \mathcal{Q}_i(x)=\tfrac{1}{2}\langle x,x \rangle_{i}\,.
\end{equation}
Let $ J_{i} \in \GL_{d_i}(F)$ be the matrix of $\langle\,,\, \rangle_{i}$:
\begin{align} \label{Ji}
\langle x,y \rangle_{i}:={}^txJ_{i}y\,.
\end{align}

Recall that for 
\begin{equation}\label{V}
  V=\prod_{i=1}^3V_i  
\end{equation}
 and $R$ an $F$-algebra we have defined
\begin{align} \label{Y}
Y(R):&=\{ (y_1,y_2,y_3) \in V(R):\mathcal{Q}_1(y_1)=\mathcal{Q}_2(y_2)=\mathcal{Q}_3(y_3)\}\,.
\end{align}
We let 
\begin{align} \label{Vprime}
    V' \subset V
\end{align}
be the open subscheme of tuples $(v_1,v_2,v_3)$ such that $v_i \neq 0$ for at least $2$ indices $i$, and, as in the introduction, set
\begin{align}\label{Ysm}
Y^{\mathrm{sm}}:=Y \cap V'\,.
\end{align}

\subsection{A symplectic similitude group} \label{ssec:simil}

Equip the module $\ZZ^{\oplus 6}$ with the alternating form 
\begin{align} \label{alt:form}
(x,y) \mapsto \sum_{i=1}^3 \left(x_iy_{i+3}-y_ix_{i+3}\right).
\end{align}
Let $\mathrm{Sp}_6$ denote the symplectic group of this form.  Concretely, for $\ZZ$-algebras $R$, we have
$$
\mathrm{Sp}_6(R):=\left\{ g \in \GL_6(R): g\left(\begin{smallmatrix} & -I_3\\I_3 &  \end{smallmatrix} \right){}^{t}g\left(\begin{smallmatrix} & I_3\\-I_3 &  \end{smallmatrix} \right)= 1
\right\}.
$$
We usually regard $\mathrm{Sp}_6$ as a group over $F$ (by base change). 

Recall that $G= \mathrm{SL}_2^3$.
We often 
identify $G(R)$ with the subgroup
$\SL_2(R^3)\leq \mathrm{Sp}_6(R)$:
\begin{align} \label{Gembed}
G(R)=\left\{\left(\begin{smallmatrix} a_1 & & &b_1& & \\ & a_2 & && b_2 & \\ & & a_3 & & & b_3\\ c_1 & & & d_1 & & \\ & c_2 & & & d_2 & \\ & & c_3 & & & d_3\end{smallmatrix}\right) \in \GL_6(R) :a_id_i-b_ic_i=1 \textrm{ for }1 \leq i \leq 3 \right\}.
\end{align}

Let $P$ be the (Siegel) parabolic subgroup of $\mathrm{Sp}_6$ whose points in an $F$-algebra $R$ are given by 
\begin{align}
\label{Siegel} P(R)=\left\{\begin{pmatrix} A & \\
&  {}^tA^{-1}\end{pmatrix}\begin{pmatrix} I_3 & Z \\ & I_3 \end{pmatrix}:A \in \GL_3(R), \quad  {}^t Z=Z \right\},
\end{align}
and let $[P,P]$ denote its commutator subgroup: 
$$
[P,P](R):=\left\{\begin{pmatrix} A & \\
&  {}^tA^{-1}\end{pmatrix}\begin{pmatrix} I_3 & Z \\ & I_3 \end{pmatrix}:A \in \SL_3(R), \quad  {}^tZ=Z \right\}.
$$
 We let $M\leq P$ be the Levi subgroup consisting of block diagonal matrices and let $N$ be the unipotent radical of $P$.

\subsection{Braverman and Kazhdan's spaces}
\label{ssec:BK}
Let
\begin{align}\label{BKspaceX}
X:=[P,P] \backslash \Sp_6\,.
\end{align}
 We note that $X$ is an $M^{\mathrm{ab}} \times \mathrm{Sp}_6$ variety (with $M^{\mathrm{ab}}:=[M,M] \backslash M$ acting on the left and $\mathrm{Sp}_6$ on the right).  
Note that this is different from the convention in \cite{BK:normalized}.  In loc.~cit.~$M^{\mathrm{ab}}$ acts on the right.  We have chosen to let it act on the left because this is the convention in the theory of Eisenstein series. 
By \cite[Lemma 2.1]{Getz:Liu:BK} the natural maps
\begin{align*}
[P,P](F) \backslash \mathrm{Sp}_6(F) \lto X(F) \quad \textrm{ and } \quad 
P(F) \backslash \mathrm{Sp}_6(F) \to P \backslash \mathrm{Sp}_6(F)
\end{align*}
are bijective.

We now compute a set of representatives for
\begin{align*}
X(F)/G(F)
\end{align*}
and the corresponding stabilizers.  We start by recalling that $P\backslash \mathrm{Sp}_6(F)$ can be viewed as the space of maximal isotropic subspaces of $F^6$ equipped with the alternating form \eqref{alt:form}.  Each such space is $3$ dimensional, so we can represent such a space by a triple of vectors in $F^6$. 
Let 
\begin{align} \label{W}
W=\langle (0,0,0,1,0,0),(0,0,0,0,1,0),(0,0,0,0,0,1)\rangle\,.
\end{align}
Then $P$ is the stabilizer of $W$.  

We consider the following maximal isotropic subspaces:
\begin{align*}
W_{0,0,0}:&=\langle (1,1,1,0,0,0),(0,0,0,-1,1,0),(0,0,0,-1,0,1) \rangle\,,\\
W_{1,0,0}:&=\langle (0,0,0,1,0,0), (0,0,0,0,1,1),(0,1,-1,0,0,0) \rangle\,,\\
W_{0,1,0}:&=\langle (0,0,0,0,1,0),(0,0,0,1,0,1),(1,0,-1,0,0,0)\rangle\,,\\
W_{0,0,1}:&=\langle (0,0,0,0,0,1),(0,0,0,1,1,0),(1,-1,0,0,0,0)\rangle\,,\\
W_{1,1,1}:=W:&=\langle (0,0,0,1,0,0),(0,0,0,0,1,0),(0,0,0,0,0,1)\rangle\,.
\end{align*}
Let
\begin{align*}
\gamma_{0,0,0}:&=\left(\begin{smallmatrix} 0 & 0 & 0 & -1 & 0& 0\\ 0 & 1 & 0 & 0 & 0 &0\\
0 & 0 & 1 & 0& 0 & 0\\
1 & 1 & 1 & 0 & 0 & 0\\ 0 & 0 & 0 &-1 & 1 &  0\\
0 & 0 & 0 & -1 & 0 & 1  \end{smallmatrix}\right),\\
(\gamma_{1,0,0},\gamma_{0,1,0},\gamma_{0,0,1}):&=\left(\left(\begin{smallmatrix} 1 & 0 & 0 & 0 & 0 & 0\\
0 & 1 & 0 & 0 & 0 &0\\
0 & 0 & 0 & 0 & 0 & 1\\
0 & 0 &0 &1 & 0 & 0
\\
0& 0 & 0 & 0 & 1 & 1\\
0 & 1 & -1 & 0 & 0 & 0\end{smallmatrix}\right),
\left(\begin{smallmatrix} 0 & 1 &0 & 0 & 0& 0\\
1 & 0 & 0 & 0 & 0 &0\\
0 & 0 & 0 & 0 & 0 & 1\\
0 & 0 &0 &0 &1 & 0\\ 
0 & 0 & 0 & 1 &0 & 1\\
1 &0 & -1 &0 & 0 & 0\end{smallmatrix}\right),\left(\begin{smallmatrix} 0 & 0 &1 & 0 & 0& 0\\
1 & 0 & 0 & 0 & 0 &0\\
0 & 0 & 0 & 0 & 1 & 0\\
0 & 0 &0 &0 &0 & 1\\ 
0 & 0 & 0 & 1 &1 & 0\\
1 &-1 & 0 &0 & 0 & 0\end{smallmatrix}\right)\right). 
\end{align*}
All four matrices are in $\mathrm{Sp}_6(\ZZ)$ and $W_{a}=W\gamma_{a}$.  We denote by $I_{a_1,a_2,a_3}$
the stabilizer in $G$ of  $W_{a_1,a_2,a_3}$.  

For $F$-algebras $R$, let
\begin{align}\label{T0N0}
\begin{split}
T_0(R):&=\left\{\left(\left(\begin{smallmatrix} a & \\ & a^{-1} \end{smallmatrix} \right),\left(\begin{smallmatrix} a & \\ & a^{-1} \end{smallmatrix} \right),\left(\begin{smallmatrix} a & \\ & a^{-1} \end{smallmatrix} \right) \right): a\in R^\times \right\},\\
N_0(R):&=\left\{\left(\left(\begin{smallmatrix} 1 & t_1 \\ & 1 \end{smallmatrix} \right),\left(\begin{smallmatrix} 1 & t_2\\ & 1 \end{smallmatrix} \right),\left(\begin{smallmatrix} 1 & t_3\\ & 1 \end{smallmatrix} \right) \right): t_i \in R, \sum_{i=1}^3t_i=0\right\}.
\end{split}
\end{align}
These are subgroups of $G$, and $T_0$ normalizes $N_0$.

\begin{lem} \label{lem:PSR:orbit}
The set $P\backslash \mathrm{Sp}_6(F)/G(F)$ has $5$ elements.  Representatives for these elements are given by the spaces $W_{a_1,a_2,a_3}$.  The stabilizers of these spaces are given as follows:  
\begin{enumerate}
\item $I_{0,0,0}=T_0N_0$.
\item $I_{1,0,0}(R)=\left\{\Big(\left(\begin{smallmatrix} a & t \\ & a^{-1} \end{smallmatrix} \right),g,\left(\begin{smallmatrix}1 & \\ & -1 \end{smallmatrix} \right)g\left(\begin{smallmatrix}1 & \\ & -1 \end{smallmatrix} \right)\Big):g \in \SL_2(R), a \in R^\times, t \in R \right\}$, 
\item $I_{0,1,0}(R):=\left\{\Big(g,\left(\begin{smallmatrix} a & t \\ & a^{-1} \end{smallmatrix} \right),\left(\begin{smallmatrix}1 & \\ & -1 \end{smallmatrix} \right)g\left(\begin{smallmatrix}1 & \\ & -1 \end{smallmatrix} \right)\Big):g \in \SL_2(R), a \in R^\times, t \in R\right\}$,
\item $I_{0,0,1}(R)=\left\{\Big(g,\left(\begin{smallmatrix}1 & \\ & -1 \end{smallmatrix} \right)g\left(\begin{smallmatrix}1 & \\ & -1 \end{smallmatrix} \right),\left(\begin{smallmatrix} a & t \\ & a^{-1} \end{smallmatrix} \right)\Big):g \in \SL_2(R), a \in R^\times, t \in R\right\}$,
\item $I_{1,1,1}=G \cap P$, the upper triangular matrices in $G$.
\end{enumerate}
\end{lem}
\begin{proof}
By \cite[Lemma 1.1]{PSRallisTriple}, if $\widetilde{P}$ denotes the parabolic subgroup of $\mathrm{GSp}_6$ containing $P$ then the given spaces $W_{a_1,a_2,a_3}$ are representatives for $\widetilde{P}\backslash \mathrm{GSp}_6(F)/\widetilde{G}(F)$, where $\widetilde{G}(F)$ is the group of $(g_1,g_2,g_3) \in \GL_2^{3}(F)$ such that $\det g_1=\det g_2=\det g_3$.  In the notation of loc.~cit., $W_{a_1,a_2,a_3} \in X_{a_1,a_2,a_3}$.  
On the other hand one checks that the natural map
$$P\backslash \Sp_6(F)/G(F) \to \widetilde{P}\backslash \mathrm{GSp}_6(F)/\widetilde{G}(F)$$
is a bijection, so the first two assertions of the lemma follow.

The assertion on the stabilizers is implicit in the corollary of \cite[Lemma 1.1]{PSRallisTriple}.  Since we have given explicit lifts $\gamma_a$ of $W_a$ under the map $\mathrm{Sp}_6(F) \to X(F)$ it is easy to verify that it is correct.
\end{proof}

\begin{lem}\label{lem:bij}  The natural map
$$
[P,P](F) \backslash \Sp_6(F)/G(F) \lto P \backslash \Sp_6(F)/G(F)
$$
is bijective.  
\end{lem}

In the remainder of the paper it is sometimes convenient to adopt the following notation:
\begin{align} \label{gammas}
\gamma_0:=\gamma_{0,0,0}, \quad  \gamma_1:=\gamma_{1,0,0}, \quad \gamma_2:=\gamma_{0,1,0},\quad \gamma_3:=\gamma_{0,0,1}\,.
\end{align}

\begin{proof}
We clearly have 
$$
[P,P](F)G(F)=P(F)G(F)\,.
$$
Moreover, for any $x \in F^\times$, 
\begin{align*}
\gamma_{0}\left(\begin{smallmatrix} xI_3 & \\ & x^{-1}I_3\end{smallmatrix}\right)\gamma_0^{-1}=\left(\begin{smallmatrix} x^{-1} & & & & &\\
& x & & & &\\ & & x & & &\\ & & & x & & \\ & & & & x^{-1} & \\ & & & & & x^{-1}\end{smallmatrix}\right),
\end{align*}
and  $\det \left(\begin{smallmatrix} x^{-1}& & \\ & x & \\ && x \end{smallmatrix} \right)=x$.
Thus
\begin{align*}
P(F) \gamma_0G(F)&=\bigcup_{x \in F^\times} [P,P](F) \left(\begin{smallmatrix} x^{-1} & & & & &\\
& x & & & &\\ & & x & & &\\ & & & x & & \\ & & & & x^{-1} & \\ & & & & & x^{-1}\end{smallmatrix}\right)\gamma_0 G(F)\\&=\bigcup_{x \in F^\times}[P,P](F)\gamma_0\left(\begin{smallmatrix} xI_3 & \\ & x^{-1}I_3\end{smallmatrix}\right)G(F)\\&=[P,P](F)\gamma_0G(F)\,.
\end{align*}
One checks similarly that $P(F)\gamma_jG(F)=[P,P](F)\gamma_jG(F)$, for $1 \leq j \leq 3$; the relevant matrix computations are below:
\begin{align*}
\gamma_1\left( \begin{smallmatrix} xI_3 & \\ & x^{-1}I_3\end{smallmatrix} \right)\gamma_1^{-1}=
\gamma_2\left( \begin{smallmatrix} xI_3 & \\ & x^{-1}I_3\end{smallmatrix} \right)\gamma_2^{-1}=
\gamma_3\left( \begin{smallmatrix} xI_3 & \\ & x^{-1}I_3\end{smallmatrix} \right)\gamma_3^{-1}=\left(\begin{smallmatrix} x & & & & &\\
& x & & & &\\ & & x^{-1} & & &\\ & & & x^{-1} & & \\ & & & & x^{-1} & \\ & & & & & x\end{smallmatrix}\right).
\end{align*}
\end{proof}

For $\gamma \in X(F)$, let $G_{\gamma} \leq G$ be the stabilizer of $\gamma$.  A simple matrix computation implies the following lemma:
\begin{lem} \label{lem:stab} One has
\begin{align}\label{Ggammas}
\begin{split}
G_{\gamma_{0}}(R):&=N_0(R)\,,\\
G_{I_3}(R):&=\left\{\left(\left(\begin{smallmatrix} b_1^{-1} & t_1 \\ & b_1 \end{smallmatrix} \right),\left(\begin{smallmatrix}b_2^{-1} & t_2 \\ & b_2 \end{smallmatrix} \right),\left(\begin{smallmatrix} b_3^{-1} & t_3 \\ & b_3 \end{smallmatrix} \right) \right):t_1,t_2,t_3 \in R, b_1,b_2,b_3 \in R^\times, b_1b_2b_3=1\right\},\\
G_{\gamma_{1}}(R):&=\left\{\Big(\left(\begin{smallmatrix} 1 & t \\ & 1 \end{smallmatrix} \right),g,\left(\begin{smallmatrix}1 & \\ & -1 \end{smallmatrix} \right)g\left(\begin{smallmatrix}1 & \\ & -1 \end{smallmatrix} \right)\Big) :t \in R, g \in \SL_2(R)\right\},\\
G_{\gamma_{2}}(R):&=\left\{\Big(g,\left(\begin{smallmatrix} 1 & t \\ & 1 \end{smallmatrix} \right),\left(\begin{smallmatrix}1 & \\ & -1 \end{smallmatrix} \right)g\left(\begin{smallmatrix}1 & \\ & -1 \end{smallmatrix} \right)\Big) :t \in R, g \in \SL_2(R)\right\},
\\
G_{\gamma_{3}}(R):&=\left\{\Big(g,\left(\begin{smallmatrix}1 & \\ & -1 \end{smallmatrix} \right)g\left(\begin{smallmatrix}1 & \\ & -1 \end{smallmatrix} \right),\left(\begin{smallmatrix} 1 & t \\ & 1 \end{smallmatrix} \right)\Big) :t \in R, g \in \SL_2(R)\right\}.
\end{split}
\end{align}
\qed
\end{lem}

\subsection{A Pl\"ucker embedding of $X$} \label{ssec:Plucker}

Let $P$ the  Siegel parabolic subgroup from above.
We can use the Pl\"ucker embedding to give a linear description of $X$.  
We construct a commutative diagram
\begin{align} \label{Pl:diag}
\begin{CD}
[P,P] \backslash \mathrm{Sp}_6 @>{\mathrm{Pl}}>> \wedge^3 \GG_a^6 -\{0\} \\ @VVV  @VVV\\
P \backslash \mathrm{Sp}_6 @>>> \mathbb{P}(\wedge^3 \GG_a^6)
\end{CD}
\end{align}
of morphisms of $F$-schemes as follows.  
The Lagrangian subspace fixed by $P$ is $W$.  For a ring $R$ and $g=\left(\begin{smallmatrix} A\\ B \end{smallmatrix}\right) \in \mathrm{Sp}_6(R)$ for $3\times 6$ matrices $A,B$   we define
\begin{align} \label{Pl}
\mathrm{Pl}(g)=b_1 \wedge b_2 \wedge b_3\,,
\end{align}
where $b_i$ is the $i$th row of $B$.  The bottom arrow just sends a point in $P \backslash \mathrm{Sp}_6$ to the line spanned by this vector.

Let $\mathrm{Sp}_{6}(F)$ act on $F^{6}$ on the right.  One obtains an induced action on $\wedge^3 F^{6}$.  For the remainder of this section assume that $F$ is a local field.  When $F$ is Archimedean let $K \leq \mathrm{Sp}_6(F)$ be a maximal compact subgroup, choose a positive definite bilinear form $(\cdot,\cdot)$ on $\wedge^3 F^{6}$ that is invariant under the action of $K$ and set $|x|=(x,x)^{[F:\RR]/2}$.  In the non-Archimedean case let $e_1,\dots,e_{6}$ be the standard basis of $F^{6}$ and let 
$$
\{e_{\alpha_1,\alpha_2,\alpha_3}:=e_{\alpha_1} \wedge e_{\alpha_2} \wedge e_{\alpha_3}:1 \leq \alpha_1<\alpha_2<\alpha_3\leq 6\}
$$
be the natural induced basis of $\wedge^3 F^{6}$.  Then
set
$$
\left|\sum_{1 \leq \alpha_{1}<\alpha_2<\alpha_3 \leq 6}x_{\alpha_1,\alpha_2,\alpha_3}e_{\alpha_1,\alpha_2,\alpha_3}\right|=\max_{1\leq \alpha_1<\alpha_2<\alpha_3 \leq 6}|x_{\alpha_1,\alpha_2,\alpha_3}|\,.
$$
This norm is invariant under the natural action of $\GL(\wedge^3\OO^6)$ on the left or right by an easy argument (see \cite[\S 2]{Getz:Liu:BK}).  Here and below $\OO$ denotes the ring of integers of a local non-Archimedean or global field $F$.
 We then set 
\begin{align} \label{norm:def}
 |g|:=|\mathrm{Pl}(g)|\,.
\end{align}  

For any $c \in \ZZ$, let
\begin{align} \label{c:def}
c(x):=\left( \begin{smallmatrix} x^{-c} & &  & & \\ & 1 & & & \\ & & 1 & & \\ & & & x^{c} & & \\ & & & & 1 & \\ & & & & & 1\end{smallmatrix} \right).
\end{align}
In this way we obtain an isomorphism $\ZZ \cong X_*(M/M^{\mathrm{der}})$; we often use this isomorphism to identify integers with cocharacters of $M/M^{\mathrm{der}}$.  
We have chosen our basis so that for non-Archimedean $F$ with uniformizer $\varpi$ one has
$|c(\varpi)| \to 0$ as $c \to \infty$.
The Iwasawa decomposition implies that 
\begin{align} \label{Iwasawa:1}
X(F)=\coprod_{c \in \ZZ} [P,P](F)c(\varpi)\mathrm{Sp}_6(\OO)
\end{align}
in the non-Archimedean case, and 
\begin{align} \label{Iwasawa:2}
X(F)=\bigcup_{t \in \RR_{>0}}[P,P](F)1(t)K
\end{align}
in the Archimedean case.  

By \cite[Proposition 2.3]{Getz:Liu:BK}, there is a continuous injection
\begin{align} \label{inj} \begin{split}
X(F)/K \lto \RR_{>0}\\
[P,P](F)gK \longmapsto |g|\,, \end{split}
\end{align} 
where $K=\mathrm{Sp}_6(\OO)$ in the non-Archimedean case.

\section{The Weil representation and theta functions} 

\subsection{The local definition of the Weil representation}\label{ssec:Weil}
In the introduction we started with a triple of quadratic spaces of even dimension over a number field $F$.  For this subsection we fix a place $v$ of $F$ which we omit from notation, writing $F:=F_v$, etc.

Let $\mathrm{O}_{\mathcal{Q}_i}$ be the orthogonal group of $\mathcal{Q}_i$.    Weil (following Segal and Shale) defined the Weil representation 
\begin{equation}\label{rholocal}
  \rho:=\rho_\psi:\mathrm{SL}_2(F) \times \mathrm{O}_{\mathcal{Q}_i}(F) \times \mathcal{S}(V_i(F)) \lto \mathcal{S}(V_i(F))\,.  
\end{equation}
Let $
\gamma(\mathcal{Q}_i)$
be the Weil number as in \cite[Th\'eor\`eme 2 and \S 24]{Weil:Certains:groupes}.
Then the representation is given on the $\mathrm{O}_{\mathcal{Q}_i}(F)$ factor by $f \mapsto (v \mapsto f(h^{-1}v))$ and on the $\mathrm{SL}_2(F)$ factor by 
\begin{enumerate}
\item $\rho\left(\begin{smallmatrix} & 1 \\ -1 & \end{smallmatrix} \right)f(v)=\gamma(\mathcal{Q}_i)\int_{V_i(F)}f(t)\psi({}^tvJ_it)dt$.
\smallskip
\item 
$\rho\left(\begin{smallmatrix} 1 & t \\ & 1 \end{smallmatrix} \right)f(v)=\psi(t\mathcal{Q}_i(v))f(v)$ for $t \in F$.\smallskip
\item $\rho\left(\begin{smallmatrix} a & \\ &a^{-1} \end{smallmatrix} \right)f(v)=(a,(-1)^{\frac{d_i}{2}}\det (J_i))|a|^{\dim_FV_i/2}f(av)$ for $a \in F^\times$.  
\end{enumerate}
\noindent  Here $dt$ is assumed to be the self-dual measure with respect to the pairing $(v,t) \mapsto \psi({}^tvJ_it)$.  A convenient reference is \cite[Chapter 2]{YYZ:book}.
The Hilbert symbol $(a,b)$ appearing in the definition above 
takes values in $\pm 1$ and is bimultiplicative.  Thus for each $i$ there are characters $\chi_{\mathcal{Q}_i}:F^\times \to \pm 1$ such that $$\chi_{\mathcal{Q}_i}(a):=(a,(-1)^{\frac{d_i}{2}}\det (J_i))\,.$$  
We write 
\begin{align} \label{chiQ}
\chi_{\mathcal{Q}}(a):=\prod_{i=1}^3\chi_{\mathcal{Q}_i}(a_i)
\end{align}
for $a \in (F^\times)^3$.  Applying the Bruhat decomposition on $\mathrm{SL}_2(F)$ we see that the information above is enough to uniquely define the representation. 

Let $\mathrm{GO}_{\mathcal{Q}_i}$ denote the similitude group of the form 
$\mathcal{Q}_i$.  Consider the semidirect product
$$
\SL_2 \rtimes \mathrm{GO}_{\mathcal{Q}_i}\,,
$$
where
$$
(g \rtimes h)(g' \rtimes h'):=g\left(\begin{smallmatrix} 1 & \\ &  \lambda(h)  \end{smallmatrix} \right)g'\left(\begin{smallmatrix}1 & \\  & \lambda(h)^{-1} \end{smallmatrix} \right) \rtimes hh'\,.
$$
For $h \in \mathrm{GO}_{\mathcal{Q}_i}(F)$ and $f \in \mathcal{S}(V_i(F))$ let
\begin{equation}\label{L(h)}
   L(h)f(v):=f(h^{-1}v)\,. 
\end{equation}
The following is \cite[Lemma 5.1.2]{HarrisKudla:Arithmetic}:
\begin{lem} \label{lem:HK}
The map
\begin{align*}
\mathrm{SL}_2(F) \rtimes \mathrm{GO}_{\mathcal{Q}_i}(F) \times \mathcal{S}(V_i(F)) \lto \mathcal{S}(V_i(F))\\
(g \rtimes h ,f) \longmapsto \rho(g)(L(h)f)
\end{align*}
defines an action of $\mathrm{SL}_2(F) \rtimes \mathrm{GO}_{\mathcal{Q}_i}(F)$ on $\mathcal{S}(V_i(F))$. \qed
\end{lem}

Strictly speaking, the definition of $L(h)$ in loc.~cit. is slightly different in that they renormalized $L(h)$ by a power of the similitude character, but this does not affect the validity of the lemma. We note in particular that the actions of $\mathrm{GO}_{\mathcal{Q}_i}(F)$ and $\mathrm{SL}_2(F)$ on $\mathcal{S}(V_i(F))$ do not commute.  

In fact, it is easy to prove Lemma \ref{lem:HK} directly from the definition of the Weil representation given the following fact:
\begin{lem} \label{lem:char:triv} Let $W$ be an even-dimensional vector space over $F$ and let $Q$ be a nondegenerate quadratic form on $W$.  Let $\Phi \in \GL_d(F)$ be the matrix of $Q$, let $\chi_Q(a):=(a,(-1)^{d/2}\det \Phi)$ and let $\mathrm{GO}_{Q}$ be the similitude group of $Q$ with similitude character $\lambda:\mathrm{GO}_Q \to \GG_m$.  Then
$$
\chi_Q(\lambda(g))=1
$$
for all $g \in \mathrm{GO}_Q(F)$.
\end{lem}
\noindent The proof of this lemma is omitted in \cite{HarrisKudla:Arithmetic} so we give it for the convenience of the reader.

\begin{proof}
By a lemma of Diedonn\'e $\lambda(g)$ is a norm from the center of the even Clifford algebra of $Q$ \cite[Lemma 13.22]{Involutions}.  This center is the quadratic \'etale $F$-algebra 
$$
F[X]/(X^2-(-1)^{d(d-1)/2}\det \Phi)=F[X]/(X^2-(-1)^{d/2}\det \Phi)
$$ 
\cite[Theorem 8.2]{Involutions} and the character attached to this quadratic \'etale $F$-algebra by local class field theory is precisely 
$(a,(-1)^{d/2}\det \Phi)$.
\end{proof}

\subsection{Theta functions} \label{ssec:Theta}

In this subsection we work globally over the number field $F$.
The global tensor product of the local representations of \S \ref{ssec:Weil} is a representation of $\mathrm{SL}_2(\A_F)$ on $\mathcal{S}(V_i(\A_F))$ and we therefore obtain a representation
\begin{equation}\label{rhoglobal}
  \rho:=\rho_\psi:G(\A_F) \times \mathcal{S}(V(\A_F)) \lto \mathcal{S}(V(\A_F))\,.  
\end{equation}

For $f \in \mathcal{S}(V(\A_F))$ and $g \in G(\A_F)$, we let 
\begin{align}\label{Theta}
\Theta_f(g):=\sum_{\gamma \in V(F)} \rho(g)f(\gamma)\,.
\end{align}
It is obvious that the sum here is absolutely convergent.  This is the usual $\Theta$ function, although we are only considering its behavior in the symplectic variable (note that $\SL_2=\mathrm{Sp}_2$).  We always take the argument of the function in the orthogonal variable to be the identity in the appropriate product of orthogonal groups.  Thus we have suppressed this variable from notation.

\section{Another space of functions} \label{sec:Schwartz}

Let $v$ be a place of the number field $F$ and let $F:=F_v$.  
In this section we start by recalling the Schwartz spaces of Braverman and Kazhdan \cite{BK:normalized}, specialized to our setting, and then apply it to construct a new space of functions that combines the space of functions in loc.~cit. with $\mathcal{S}(V(F))$.  We should point out that the papers \cite{Shahidi:LF,Shahidi:FT} provide valuable additional information about Braverman and Kazhdan's Schwartz spaces.

\subsection{Schwartz spaces} \label{ssec:Schwartz:spaces}
Let $K \leq \mathrm{Sp}_6(F)$ be a  maximal compact subgroup that is conjugate to $\mathrm{Sp}_6(\OO)$ if $F$ is non-Archimedean.
 In \cite{Getz:Liu:BK} the authors defined a Schwartz space $\mathcal{S}_{BK}(X(F),K)$ of functions on $X(F)$ roughly following the approach of Braverman and Kazhdan.  
Functions in $\mathcal{S}_{BK}(X(F),K)$ are smooth and $K$-finite under the natural right action of $K$ on $X(F)$.   
We recall the growth properties of these functions in this section.
 
Recall that the norm of $x \in X(F)$ is defined in \eqref{norm:def}.  The following is \cite[Lemmas 5.1 and 5.7]{Getz:Liu:BK}:
 
\begin{lem} \label{lem:bounded} Let
 $g \in \mathrm{Sp}_6(F)$ and $\Phi \in \mathcal{S}_{BK}(X(F),K)$.  If $F$ is non-Archimedean one has
$$
|\Phi(g)|_{\mathrm{st}} \ll_{\Phi} |g|^{-2}\,.
$$
The support of $\Phi$ is contained in 
$$
\bigcup_{c>-N}[P,P](F)c(\varpi)\mathrm{Sp}_6(\OO)\,,
$$
for sufficiently large $N$ (depending on $\Phi$). 
If $F$ is Archimedean for any $N \in \ZZ_{\geq 0}$  
one has
$$
|\Phi(g)| \ll_{\Phi,N}|g|^{-2-N}\,.
$$ 
\qed
\end{lem}

For $F$ non-Archimedean {define}
\begin{align}\label{b}
b(g)=\sum_{j=0}^\infty \sum_{k=0}^\infty q^{2j}\one_{[P,P](F)(k+2j)(\varpi)\mathrm{Sp}_6(\OO)}(g) \in \mathcal{S}_{BK}(X(F),\mathrm{Sp}_6(\OO))\,,
\end{align}
where $q$ is the cardinality of the residue field.
The following is \cite[Lemma 5.3]{Getz:Liu:BK}:
\begin{lem} \label{lem:basic:bound}
Assume that $F$ is non-Archimedean.  Let $\varepsilon>0$. For $q$  sufficiently large in a sense depending on $\varepsilon$ one has
$$
|b(g)| \leq  |g|^{-2-\varepsilon}\,.
$$
\qed
\end{lem}
Let $\psi:F \to \CC^\times$ be a nontrivial character.  In loc.~cit. we also defined a Fourier transform
\begin{align}\label{mathcalF}
\mathcal{F}:=\mathcal{F}_{BK,\psi}:\mathcal{S}_{BK}(X(F),K)  \lto \mathcal{S}_{BK}(X(F),K)\,.
\end{align}
{Assume $F$ is non-Archimedean and $\psi$ is unramified. Then the function $b$ of \eqref{b} enjoys the following three properties: 
\begin{enumerate}
\item $b(xk)=b(x)$ for all $(x,k) \in X(F) \times \Sp_6(\OO)$,
    \item $\mathcal{F}(b)=b$ (see \cite[Lemma 5.4]{Getz:Liu:BK}),
    \item The support of $b$ is integral in the sense that it is mapped to elements of $\wedge^n\OO^{2n}$ under the Pl\"ucker embedding $\mathrm{Pl}$ of \eqref{Pl:diag}.
\end{enumerate}
Because of this we refer to $b$ as the \textbf{basic function} in $\mathcal{S}(X(F),\Sp_6(\OO))$.
  Using ``the'' is an abuse of language because 
  the conditions above do not specify $b$ uniquely.  For example any scalar multiple of $b$ would also satisfy these conditions.  However it is a convenient abuse of language that we will continue to use.  
}

\subsection{Local functions} \label{ssec:loc:func}

For $(f_1,f_2) \in \mathcal{S}_{BK}(X(F),K) \times \mathcal{S}(V(F))$ let
\begin{align} \label{Is}\begin{split}
I(f_1,f_2)\left(v \right)&=\int_{N_0(F) \backslash G(F)} f_1\left(\gamma_{0}g\right) \rho\left(g\right)f_2(v)dg, \quad v \in Y^{\mathrm{sm}}(F)\,.
\end{split}
\end{align}
This is the local factor of the integral one obtains after unfolding the integral of our theta function $\Theta_{f_2}$ against $\sum_{\gamma \in X(F)}f_1(\gamma g)$ as explained informally after \eqref{basic:id}.  The full argument is given in the proof of Theorem \ref{thm:main} below.  
It is interesting to note that the integral is not well-defined if one tries to evaluate it at a general $v \in V'(F)$ because the function $\rho (g)f_2(v)$ is only left invariant under $N_0(F)$ for $v \in Y(F)$.  {However, the integral 
\begin{align} \label{Isabs}\begin{split}
&\int_{N_0(F) \backslash G(F)} |f_1\left(\gamma_{0}g\right) \rho\left(g\right)f_2(v)|dg, \quad v \in V'(F)\,.
\end{split}
\end{align}
is well-defined because $|\rho (g)f_2(v)|$ is left invariant under $N_0(F)$.}

In \S \ref{sec:unr} we will compute \eqref{Is} in the unramified case, and in \S \ref{sec:bound:na} and \S \ref{sec:bound:arch} we will bound it {by bounding \eqref{Isabs}} in the non-Archimedean and Archimedean cases, respectively.

\subsection{A transform} \label{ssec:transform}

Consider the transform
\begin{align*}
I(f_1,f_2) &\longmapsto I(\mathcal{F}(f_1),f_2)\,.
\end{align*}
It can profitably be viewed as a sort of Fourier transform.
\begin{rem} {If $F$ is non-Archimedean, $\psi$ is unramified, the matrices $J_i$ defining the $\mathcal{Q}_i$ are in $\GL_{d_i}(\OO)$, and $\rho(k)\one_{V(\OO)}=\one_{V(\OO)}$ for all $k \in \SL_2^3(\OO)$  then $I(b,\one_{V(\OO)})$ can be thought of as a basic function:
\begin{enumerate}
\item $I(b,\one_{V(\OO)})(k^{-1}v)=I(b,\one_{V(\OO)})(v)$ for $(k,v) \in H(\OO) \times Y(\OO)$ (see Proposition \ref{prop:I0comp}),
\item The function $I(b,\one_{V(\OO)})$ is invariant under the transform  $I(f_1,f_2) \mapsto I(\mathcal{F}(f_1),f_2)$,
\item The support of $I(b,\one_{V(\OO)})$ is contained in $V(\OO) \cap Y(\OO)$ (see Proposition \ref{prop:I0comp}).  
\end{enumerate}
\noindent 
Here we have given $V=\prod_{i=1}^3\GG_a^{d_i}$ the evident structure of a scheme over $\OO$, given $Y$ the structure of a scheme over $\OO$ by taking the schematic closure of $Y_{F}$ in $V$, and given $H$ the evident structure of a group scheme over $\OO$ using the assumption that the $J_i$ are in $\GL_{d_i}(\OO)$.} 
\end{rem}

We now compute {the behavior of the transform under the group $H$ in \eqref{H}}.
For $h \in H(F)$ let 
\begin{align} \label{Lambdah}
\Lambda(h):=\left(\begin{smallmatrix}I_3 & \\  & \lambda(h)I_3 \end{smallmatrix} \right),
\end{align}
where $\lambda$ is the similitude norm in \eqref{lambda}.
For $F$-algebras $R$ let
\begin{align}\label{omega}
\begin{split}
\omega:M(R) &\lto R^\times\\
\left(\begin{smallmatrix} A & \\ & {}^tA^{-1} \end{smallmatrix}\right) &\longmapsto \det A\,.
\end{split}
\end{align}
For $\chi:F^\times \to \CC^\times$ a character and $s \in \CC$ let 
$\chi_s:=\chi|\cdot|^s$.  For
{$f \in \mathcal{S}(X(F),K)$ and $g \in  \mathrm{Sp}_6(F)$ let
\begin{equation}\label{fchis:local}
 f_{\chi_s}(g):=\int_{ M^{\mathrm{ab}}(F)}\delta_P(m)^{1/2}\chi_{s}\left( \omega( m) \right)
f(m^{-1}g)\,dm.   
\end{equation}
This converges for $\mathrm{Re}(s)$ sufficiently large and admits a meromorphic continuation to the $s$ plane for each fixed $g$, see \cite[\S 4]{Getz:Liu:BK}.}

For functions $f$ on $V(F)$ let $L(h)f(v):=f(h^{-1}v)$.

\begin{lem} \label{lem:equiv}
Let $(f_1,f_2) \in \mathcal{S}_{BK}(X(F),K) \times \mathcal{S}(V(F))$ and $h \in H(F)$.  Let
$$
\widetilde{f}_1(g):=f_1\left(\gamma_{0}\Lambda(h)^{-1}\gamma_0^{-1} g\Lambda(h)\right).
$$
Then $\widetilde{f}_1 \in \mathcal{S}_{BK}\left(X(F),\Lambda(h)K\Lambda(h)^{-1}\right)$ and the following equalities hold:
\begin{align*}
L(h)I(f_1,f_2)&=|\lambda(h)|^{-2}I(\widetilde{f}_1,L(h)f_2)\,,\\
    I(\mathcal{F}(\widetilde{f}_1),L(h)f_2)&=|\lambda(h)|^{\sum_{i=1}^3d_i/2}L\left(\frac{h}{\lambda(h)}\right)I(\mathcal{F}(f_1),f_2)\,.
\end{align*}
\end{lem}
\begin{proof}
Using Lemma \ref{lem:HK} we have
\begin{align*}
L(h)I(f_1,f_2)\left(v \right)&=\int_{N_0(F) \backslash G(F)} f_1\left(\gamma_{0}g\right)
L(h)\rho(g)f_2(v)dg\\
&=\int_{N_0(F) \backslash G(F)} f_1\left(\gamma_{0}g\right) \rho\left(\left(\begin{smallmatrix} 1 & \\ &  \lambda(h) \end{smallmatrix} \right)g\left(\begin{smallmatrix}1 & \\  & \lambda(h)^{-1} \end{smallmatrix} \right)\right)L(h)f_2(v)dg\\
&=|\lambda(h)|^{-2}\int_{N_0(F) \backslash G(F)} f_1\left(\gamma_{0}\Lambda(h)^{-1}g\Lambda(h)\right) \rho\left(g\right)L(h)f_2(v)dg\,.
\end{align*}
To show $\widetilde{f}_1 \in \mathcal{S}_{BK}\left(X(F),\Lambda(h)K\Lambda(h)^{-1}\right)$ it suffices to check that for each character $\chi:F^\times \to \CC^\times$ the section
$\widetilde{f}_{1\chi_s}$
is excellent in the sense of \cite[\S 3]{Getz:Liu:BK}.
Since $\gamma_{0}\Lambda(h)^{-1}\gamma_0^{-1}$ normalizes $M(F)$ and $f_{\chi_s}$ is an excellent section by definition of $\mathcal{S}_{BK}(X(F),K)$, this is obvious. 

To complete the proof of the lemma we must compute
 $\mathcal{F}(\widetilde{f}_1)$.
Let 
$$
w_0:=\left(\begin{smallmatrix} & & & & & -1 \\ & & & & -1 & \\ & & & -1 & & \\ & & 1 & & & \\ & 1 & & & &\\ 1 & & & & &\end{smallmatrix} \right).
$$
Using the notation of \cite[\S 3]{Getz:Liu:BK} we compute
\begin{align} \label{intertwine:comp}
M_{w_0}\widetilde{f}_{1\chi_s}(g):&=
\int_{N(F)}\int_{M^{\mathrm{ab}}(F)}\delta_P^{1/2}(m)\chi_s(\omega(m))f_1\left(\gamma_{0}\Lambda(h)^{-1}\gamma_0^{-1}m^{-1} w_0^{-1}ng\Lambda(h)\right)dmdn\,.
\end{align}
Here we take $\mathrm{Re}(s)$ large to ensure convergence.  
One has 
\begin{align*}
[M,M](F)\gamma_{0}\Lambda(h)^{-1}\gamma_0^{-1}m^{-1}w_0^{-1}n&=[M,M](F)m^{-1}\gamma_{0}\Lambda(h)^{-1}\gamma_0^{-1}w_0^{-1}n\\
&=[M,M](F)m^{-1}w_0^{-1} (w_0\gamma_0\Lambda(h)^{-1}\gamma_0^{-1}w_0^{-1})n\,.
\end{align*}
We have
\begin{align*}
w_0\gamma_0\Lambda(h)^{-1}\gamma_0^{-1}w_0^{-1}=\left(\begin{smallmatrix} \lambda(h)^{-1} & & & & &\\ & \lambda(h)^{-1} & & & &\\ & & 1 & & &\\ & & & 1 & & \\ & & & & 1 & \\ & & & & & \lambda(h)^{-1} \end{smallmatrix} \right).
\end{align*}
Thus taking a change of variables $n \mapsto (w_0\gamma_0\Lambda(h)^{-1}\gamma_0^{-1}w_0^{-1})^{-1}n(w_0\gamma_0\Lambda(h)^{-1}\gamma_0^{-1}w_0^{-1})$, we see that
\eqref{intertwine:comp} is 
\begin{align}
|\lambda(h)|^{2}M_{w_0}f_{1\chi_s}\left((w_0\gamma_0\Lambda(h)^{-1}\gamma_0^{-1}w_0^{-1})g\Lambda(h)\right).
\end{align}
Now
\begin{align}
\left(\begin{smallmatrix} \lambda(h) & & & & &\\ & 1 & & & &\\ & & \lambda(h)^{-1} & & &\\ & & & \lambda(h)^{-1} & & \\ & & & & 1 & \\ & & & & & \lambda(h) \end{smallmatrix} \right)\left(\begin{smallmatrix} \lambda(h)^{-1} & & & & &\\ & \lambda(h)^{-1} & & & &\\ & & 1 & & &\\ & & & 1 & & \\ & & & & 1 & \\ & & & & & \lambda(h)^{-1} \end{smallmatrix} \right)=\left(\begin{smallmatrix} 1 & & & & &\\ & \lambda(h)^{-1} & & & &\\ & & \lambda(h)^{-1} & & &\\ & & & \lambda(h)^{-1} & & \\ & & & & 1 & \\ & & & & & 1 \end{smallmatrix} \right).
\end{align}

Thus by \cite[Theorem 4.4]{Getz:Liu:BK}
\begin{align} \label{FTtilde}
|\lambda(h)|^{2}\mathcal{F}(f_1)\left(\left(\begin{smallmatrix} 1 & & & & &\\ & \lambda(h)^{-1} & & & &\\ & & \lambda(h)^{-1} & & &\\ & & & \lambda(h)^{-1} & & \\ & & & & 1 & \\ & & & & & 1 \end{smallmatrix} \right)g\Lambda(h)\right)=\mathcal{F}(\widetilde{f}_1)(g)\,.
\end{align}
Hence
\begin{align*}
\begin{split}
&\ I(\mathcal{F}(\widetilde{f}_1),L(h)f_2)\\
=&\ |\lambda(h)|^{2}\int_{N_0(F) \backslash G(F)} \mathcal{F}(f_1)\left(\left(\begin{smallmatrix} 1 & & & & &\\ & \lambda(h)^{-1} & & & &\\ & & \lambda(h)^{-1} & & &\\ & & & \lambda(h)^{-1} & & \\ & & & & 1 & \\ & & & & & 1 \end{smallmatrix} \right)\gamma_{0}g\Lambda(h)\right) \rho\left(g\right)L(h)f_2(v)dg\\
=&\ |\lambda(h)|^{2}\int_{N_0(F) \backslash G(F)} \mathcal{F}(f_1)\left(\gamma_0\left(\begin{smallmatrix} \lambda(h)^{-1}I_3 &\\ && I_3 \end{smallmatrix} \right)g\Lambda(h)\right) \rho\left(g\right)L(h)f_2(v)dg\\
=&\ \int_{N_0(F) \backslash G(F)} \mathcal{F}(f_1)\left(\gamma_0g\right) \rho\left(\left(\begin{smallmatrix} \lambda(h) & \\ & 1 \end{smallmatrix} \right)g\left(\begin{smallmatrix} 1 & \\ & \lambda(h)^{-1} \end{smallmatrix} \right)\right)L(h)f_2(v)dg\,.
\end{split}
\end{align*}
By Lemma \ref{lem:HK} this is equal to 
\begin{align*}\begin{split}
&\ \int_{N_0(F) \backslash G(F)} \mathcal{F}(f_1)\left(\gamma_0g\right) L(h)\rho\left(\left(\begin{smallmatrix} \lambda(h) & \\ & \lambda(h)^{-1} \end{smallmatrix} \right)g\right)f_2(v)dg\\
=&\ |\lambda(h)|^{\sum_{i=1}^3d_i/2}\chi_{\mathcal{Q}}(\lambda(h))\int_{N_0(F) \backslash G(F)} \mathcal{F}(f_1)\left(\gamma_0g\right) L(\lambda(h)^{-1}h)\rho\left(g\right)f_2(v)dg\,.
\end{split}
\end{align*}
By Lemma \ref{lem:char:triv} $\chi_{\mathcal{Q}}(\lambda(h))=1$ and this completes the proof.
\end{proof}

\section{The summation formula} \label{sec:summation}

  Our goal in this section is to state the main theorem of this paper, Theorem \ref{thm:main}, and prove it modulo some convergence statements {and a vanishing statement} that will be established in the remainder of the paper.  Theorem \ref{thm:main} was stated in the introduction as Theorem \ref{thm:intro}. 
  Before we do this we restate the Poisson summation formula obtained in \cite{Getz:Liu:BK} using the argument of Braverman and Kazhdan.  

In this section $F$ is a number field.  Let $K:=\prod_v K_v \leq \mathrm{Sp}_6(\A_F)$ be a maximal compact subgroup such that $K^{\infty}$ is $\mathrm{Sp}_{6}(\A_F^\infty)$-conjugate to $\mathrm{Sp}_{6}(\widehat{\OO})$.   We let
$$
\mathcal{S}_{BK}(X(\A_F),K) 
$$ 
be the restricted tensor product of the local spaces $\mathcal{S}_{BK}(X(F_v),K_v)$ with respect to the basic functions $b_v$ for $v \nmid \infty$ (see \eqref{b}).

For algebraic groups $Q$ over $F$ let $[Q]:=Q(F) \backslash Q(\A_F)$.
For $f \in \mathcal{S}_{BK}(X(\A_F),K)$, a Hecke character $\chi:[\GG_m] \lto \CC^\times$ and $s \in \CC$, let $\chi_s:=\chi|\cdot|^s$ where $|\cdot|$ is the idelic norm, and let
\begin{align}\label{fchis:global}
f_{\chi_s}(g):=\int_{ M^{\mathrm{ab}}(\A_F)}\delta_P(m)^{1/2}\chi_{s}\left( \omega( m) \right)f(m^{-1}g)\,
dm\,,
\end{align}
for all $g \in \mathrm{Sp}_6(\A_F)$.
We then form the Eisenstein series 
\begin{align}\label{Eisensteinseries}
E(g;f_{\chi_s}):=\sum_{\gamma \in P(F) \backslash \mathrm{Sp}_6(F)} f_{\chi_s}(\gamma g)\,.
\end{align}
By Langlands' general theory this Eisenstein series admits a meromorphic continuation to the plane.  The possible poles of $E(g;f_{\chi_s})$ were computed in \cite{Ikeda:poles:triple}.  The poles, if they exist, are simple.  The Eisenstein series is holomorphic if $\chi^2 \neq 1$.  If $\chi=1$ there are possible poles at $s=\pm 1, s =\pm 2$, and if $\chi \neq 1$ but $\chi^2=1$ there are possible poles at $s= \pm 1$.  

Let $\kappa_F:=\mathrm{Res}_{s=1}\zeta_F(s)$. The following is \cite[Theorem 6.7]{Getz:Liu:BK}:
\begin{thm} \label{thm:BK} Let $f \in \mathcal{S}_{BK}(X(\A_F),K)$.  For every $g \in \mathrm{Sp}_6(\A_F)$ one has
\begin{align*}
&\sum_{\gamma \in X(F)}f(\gamma g)+
\frac{1}{\kappa_F}\sum_{
i=1}^2\mathrm{Res}_{s=  i}E(g;\mathcal{F}(f)_{1_s})+\frac{1}{\kappa_F}\sum_{\substack{\chi \in \widehat{[\GG_m]} \\\chi \neq 1, \chi^2=1}}
\mathrm{Res}_{s=1 }E(g;\mathcal{F}(f)_{\chi_s})\\
& = \sum_{\gamma \in X(F)} \mathcal{F}(f)(\gamma g)+
\frac{1}{\kappa_F}\sum_{i=1}^2
\mathrm{Res}_{s=i }E(g;f_{1_s})+\frac{1}{\kappa_F}\sum_{\substack{\chi \in \widehat{[\GG_m]} \\\chi \neq 1, \chi^2=1}}
\mathrm{Res}_{s=1}E(g;f_{\chi_s})\,.
\end{align*}
All of the sums here are absolutely convergent.  \qed
\end{thm}

In view of the theorem the following assumption
on a function $f \in \mathcal{S}_{BK}(X(\A_F),K)$ is natural:
\begin{align} \label{A:BK}\begin{split}
&\textrm{ One has }\mathrm{Res}_{s=1}E(g;f_{\chi_s})=0
\textrm{ when }\chi \textrm{ is a quadratic or trivial }\\&\textrm{ character in }\widehat{[\GG_m]} \textrm{ and }\mathrm{Res}_{s=2}E(g;f_{1_s})=0.\end{split}
\end{align}
We note that it is easy to find functions $f$ satisfying the assumption \eqref{A:BK}, see Theorem \ref{thm:A:BK} below.

\begin{cor} \label{cor:BK}
{Let $f \in \mathcal{S}_{BK}(X(\A_F),K)$.  Assume that $f$ and $\mathcal{F}(f)$ satisfy \eqref{A:BK}.}  Then for all $g \in \mathrm{Sp}_6(\A_F)$
\begin{align*}
&\sum_{\gamma \in X(F)}f(\gamma g)=\sum_{\gamma \in X(F)} \mathcal{F}(f)(\gamma g)\,.
\end{align*}
\qed
\end{cor}

Let $v$ be a place of $F$.  We will require the following  assumption on $f \in \mathcal{S}_{BK}(X(\A_F),K)$:
\begin{align} \label{A:vanish} 
 \textrm{ There is a place }v \textrm{ of }F \textrm{ such that  } f=f_vf^v \textrm{ and }{
 f_{v} \in C_c^\infty( \gamma_0G(F_{v}))}.
\end{align}
We will also require the following assumption on $f \in \mathcal{S}(V(\A_F))$:
\begin{align} \label{vanish}
\textrm{One has }\rho(g)f(\xi) = 0 \textrm{ for all }g \in \SL_2^3(\A_F), \xi \notin V'(F).
\end{align}
Here $V'$ is defined as in \eqref{Vprime}.
Using the fact that the Fourier transform $\mathcal{F}$ is an isomorphism \cite[Lemma 4.6]{Getz:Liu:BK} and that $K_v$-finite compactly supported functions on $X(F)$ are contained in $\mathcal{S}_{BK}(X(F_v),K_v)$ \cite[Proposition 4.7]{Getz:Liu:BK}, it is easy to find functions $f \in \mathcal{S}_{BK}(X(\A_F),K)$ such that both $f$ and $\mathcal{F}(f)$ satisfy \eqref{A:vanish}.
{In practice one can ensure \eqref{vanish} is valid as follows.  Let 
\begin{align} \label{Weyl}
    W \leq \SL_2^3(\ZZ)
\end{align}
be group of order $8$ generated by the three matrices that are $\left(\begin{smallmatrix} & 1\\ -1 & \end{smallmatrix}\right)$ in the $i$th factor and the identity in the other factors.  
 Then by the explicit description of the action of the Weil representation we see that \eqref{vanish} is implied by the following condition:
\begin{align}
    \label{vanish2} \textrm{There is a place }v\textrm{ of }F \textrm{ such that }f=f_vf^v\textrm{ and } \mathrm{sup}(\rho(w)f_v) \subseteq V'(F_v) \textrm{ for all }w \in W.
\end{align}
}

The main theorem of this paper is the following:

\begin{thm} \label{thm:main}
For 
$$
(f_1,f_2) \in \mathcal{S}_{BK}(X(\A_F),K) \times \mathcal{S}(V(\A_F))
$$ 
such that $f_1$, $\mathcal{F}(f_1)$ satisfy \eqref{A:vanish} and  $f_2$  satisfies \eqref{vanish}, one has
\begin{align*}
\sum_{\xi \in Y^{\mathrm{sm}}(F)} I(f_1,f_2)(\xi)=\sum_{\xi \in Y^{\mathrm{sm}}(F)} I(\mathcal{F}(f_1),f_2)(\xi)\,.
\end{align*} 
\end{thm}
Here for $\xi \in Y^{\mathrm{sm}}(F)$,
\begin{align} \label{Is:global}
\begin{split}
I(f_1,f_2)\left(\xi \right)&=\int_{N_0(\A_F) \backslash G(\A_F)} f_1\left(\gamma_{0}g\right) \rho\left(g\right)f_2(\xi)\,dg\,.
\end{split}
\end{align}

We will prove the theorem in this section assuming the
absolute convergence statement given in Proposition \ref{prop:1cell} {and Theorem \ref{thm:A:BK}}.  We will indicate precisely when they are invoked.  After this section, the majority of the remainder of the paper is devoted to proving {Proposition \ref{prop:1cell}}.  

One has
\begin{align} \label{formal:comp}
\nonumber &\ \int_{G(F) \backslash G(\A_F)} \sum_{\gamma \in X(F)} f_1(\gamma g)\Theta_{f_2}(g)dg \\
\nonumber =& \ \sum_{\gamma \in X(F)/G(F)} \int_{G_{\gamma}(F) \backslash G(\A_F)} f_1(\gamma g)\Theta_{f_2}(g)dg\\
= &\ \sum_{\gamma_a} \int_{G_{\gamma_{a}}(\A_F) \backslash G(\A_F)}f_1(\gamma_{a}g)\int_{[G_{\gamma_{a}}]}\Theta_{f_2}(g_1g)dg_1
dg\,,
\end{align}
where the sum is over a set of representatives for 
$X(F)/G(F)$.  By assumption \eqref{A:vanish} only the contribution of $\gamma_a=\gamma_0$ is nonzero. 
  The stabilizer $G_{\gamma_0}$ is $N_0$ (see Lemma \ref{lem:stab} and \eqref{N0}).
Using part (2) in the definition of the Weil representation one has
\begin{align*}
\int_{[N_0]}
\sum_{\xi \in V(F)}
\rho(ng)f_2( \xi)dn=\sum_{\xi \in Y^{\mathrm{sm}}(F)}
\rho(g)f_2(\xi)\,.
\end{align*}
Here we have used assumption \eqref{vanish}.
It is permissible to switch the sum and integral here 
because $f_2$ is Schwartz.  

Thus 
\begin{align*}
& \ \int_{G_{\gamma_{0}}(\A_F) \backslash G(\A_F)} f_1(\gamma_{0}g)\int_{[G_{\gamma_{0}}]}\Theta_{f_2}(g_1g)dg_1dg\\
=&\ \int_{N_0(\A_F) \backslash G(\A_F)} f_1(\gamma_{0}g)\int_{[N_0]}
\sum_{\xi \in V(F)}
\rho(ng)f_2( \xi)dn dg\\
=&\ \int_{N_0(\A_F) \backslash G(\A_F)}f_1(\gamma_{0}g)\sum_{\xi \in Y^{\mathrm{sm}}(F)}
\rho(g)f_2(\xi) dg\\
=&\ \sum_{\xi \in Y^{\mathrm{sm}}(F)}I(f_1,f_2)(\xi)\,.
\end{align*}
These formal manipulations are justified by Proposition \ref{prop:1cell} and the Fubini-Tonelli Theorem.

We have shown that 
\begin{align*}
&\ \int_{G(F) \backslash G(\A_F)} \sum_{\gamma \in X(F)} f_1(\gamma g)\Theta_{f_2}(g)dg= \ \sum_{\xi \in Y^{\mathrm{sm}}(F)} I(f_1,f_2)(\xi)\,.
\end{align*}
{Since $f_1$ and $\mathcal{F}(f_1)$ satisfy \eqref{A:vanish} we deduce from Theorem \ref{thm:A:BK} that they both satisfy \eqref{A:BK}.  Thus by }Corollary \ref{cor:BK} the integral here is
\begin{align*}
\int_{G(F) \backslash G(\A_F)} \sum_{\gamma \in X(F)} \mathcal{F}(f_1)(\gamma g)\Theta_{f_2}(g)dg\,.
\end{align*}
Replacing $f_1$ by $\mathcal{F}(f_1)$ in the argument above we see that this is 
\begin{align*}
&\sum_{\xi \in Y^{\mathrm{sm}}(F)} I(\mathcal{F}(f_1),f_2)(\xi)\,.
\end{align*}
Thus assuming the absolute convergence statement in Proposition \ref{prop:1cell} {and Theorem \ref{thm:A:BK}} we have proven Theorem \ref{thm:main}. \qed

\begin{cor} 
\label{cor:main}  Let $h \in H(\A_F)$.  For 
$$
(f_1,f_2) \in \mathcal{S}_{BK}(X(\A_F),K) \times \mathcal{S}(V(\A_F))
$$ 
such that $f_1$, $\mathcal{F}(f_1)$ satisfy \eqref{A:vanish} and  $f_2$  satisfies \eqref{vanish}, one has
\begin{align*}
\sum_{\xi \in Y^{\mathrm{sm}}(F)} I(f_1,f_2)(h^{-1}\xi)=\sum_{\xi \in Y^{\mathrm{sm}}(F)} |\lambda(h)|^{\sum_{i=1}^3d_i/2-2} I(\mathcal{F}(f_1),f_2)(\lambda(h) h^{-1}\xi)\,.
\end{align*} 
\end{cor}
\begin{proof}
In view of Theorem \ref{thm:main} and Lemma \ref{lem:equiv} it suffices to check that if $f_1$ and $\mathcal{F}(f_1)$ satisfy \eqref{A:vanish} then $\widetilde{f}_1$ and $\mathcal{F}(\widetilde{f}_1)$ satisfy \eqref{A:vanish}, where 
$$
\widetilde{f}_1(g):=f_1(\gamma_0\Lambda(h)^{-1}\gamma_0^{-1}g \Lambda(h))\,.
$$
(see \eqref{Lambdah} for the definition of $\Lambda(h)$).  We recall from \eqref{FTtilde} that 
$$
|\lambda(h)|^{2}\mathcal{F}(f_1)\left(\left(\begin{smallmatrix} 1 & & & & &\\ & \lambda(h)^{-1} & & & &\\ & & \lambda(h)^{-1} & & &\\ & & & \lambda(h)^{-1} & & \\ & & & & 1 & \\ & & & & & 1 \end{smallmatrix} \right)g\Lambda(h)\right)=\mathcal{F}(\widetilde{f}_1)(g)\,.
$$
 Since $\Lambda(h_v)$ normalizes $G(F_v)$ for all $v$ it follows that if $f_1$ satisfies \eqref{A:vanish} then so does $\widetilde{f}_1$.  
Since {
$$
\left(\begin{smallmatrix} 1 & & & & &\\ & \lambda(h_v) & & & &\\ & & \lambda(h_v) & & &\\ & & & \lambda(h_v) & & \\ & & & & 1 & \\ & & & & & 1 \end{smallmatrix} \right)\gamma_0=\gamma_0\left(\begin{smallmatrix} \lambda(h_v)I_3 & \\ & I_3\end{smallmatrix} \right), \textrm{ and }\left(\begin{smallmatrix} \lambda(h_v)I_3 & \\ & I_3\end{smallmatrix} \right)G(F_v)\Lambda(h_v)^{-1} =G(F_v)
$$}
 for all $v$, if $\mathcal{F}(f_1)$ satisfies \eqref{A:vanish} then so does $\mathcal{F}(\widetilde{f}_1)$.
\end{proof}

\section{The unramified calculation} \label{sec:unr} 
  
For this section $F$ is a local field of residual characteristic $p$ with ring of integers $\OO$ that is unramified over $\QQ_p$.  We let $\psi:F \to \CC^\times$ be an unramified nontrivial character and we assume that $\chi_{\mathcal{Q}}$ is unramified.
 To ease notation let 
$$
K:=\mathrm{Sp}_6(\OO)\,.
$$
  For $c \in X_*(M/M^{\mathrm{der}})=\ZZ$ let
\begin{align}\label{onec}
\one_c:=\one_{[P,P](F)c(\varpi)K}
\end{align}
(see \eqref{c:def}).

The following is a consequence of the Iwasawa decomposition:
\begin{lem}
The functions $\one_{c}$, $c \in \ZZ$, form a basis of $C_c^\infty(X(F)/K)$ as a $\CC$-vector space.   \qed
\end{lem}

In view of the injection \eqref{inj} we have the following lemma:
\begin{lem} \label{lem:Iwasawa} 
One has
$$
\one_{c}(g) \neq 0
$$
if and only if $|g|=q^{-c}$. \qed
\end{lem}
We recall that the basic function, by definition \eqref{b}, is
$$
b:=\sum_{j,k=0}^\infty q^{2j}\one_{k+2j}\,.
$$

In this section we compute the function
$I(b,\one_{V(\OO)})(v)$
and then give bounds on it.  Technically speaking the bounds should be proven first to ensure the absolute convergence of the integrals with which we are working.  However we feel that giving the formal computation first and then proving absolute convergence makes the argument easier to follow.  

Let
\begin{align}\label{Q}
\mathcal{Q}(v):=\sum_{i=1}^3\mathcal{Q}_i(v_i)\,.
\end{align}
{We assume that $\rho(k)\one_{V(\OO)}=\one_{V(\OO)}$ for all $k \in \SL_2^3(\OO)$. If we begin with global objects this will be true for the corresponding local objects at almost all places.}

\begin{prop} \label{prop:I0comp} Assume that $v \in Y^{\mathrm{sm}}(F)$.
The integral $I(b,\one_{V(\OO)})(v)$ is equal to
\begin{align*}
&\sum_{j=0}^\infty \int\one_{\OO}\left(\frac{\mathcal{Q}(v)}{\varpi^{4j}a_1a_2a_3}
\right)
\one_{V(\OO)}\left(\frac{v}{\varpi^{2j}a}\right)\bar{\chi}_\mathcal{Q}(\varpi^{2j}a)\prod_{i=1}^3\left(\frac{|a_i|}{q^{2j}}\right)^{1-d_i/2}d^\times a\,,
\end{align*}
where the integral is over the set of  $a \in (\OO \cap F^\times)^3$ such that 
\begin{align*}
\max(|a_1^{-1}a_2a_3|,|a_2^{-1}a_1a_3|,|a_3^{-1}a_1a_2|)\leq 1 \,.
\end{align*}
In particular $I(b,\one_{V(\OO)})$ is supported in $V(\OO)$.
\end{prop}

\begin{proof}

Let
{
\begin{align} \label{T}
T \leq G
\end{align}}be the maximal torus of diagonal matrices.  We use the Iwasawa decomposition to write
$$
dg=\frac{dn_0dndadk}{\delta_{P \cap G}(a)}\,,
$$
where $dn_0$, $dn$, $da$ and $dk$ are Haar measures on $N_0(F)$, $\{\left(\begin{smallmatrix} 1 & x \\ & 1 \end{smallmatrix} \right): x \in F\}$, $T(F)$ and $K$, respectively.  We assume that $K$ and its intersections with the other subgroups here have measure $1$.
Then we obtain
\begin{align} \label{Ib}
&\ I(b,\one_{V(\OO)})\left(v \right)\\ \nonumber = & \ \int_{N_0(F) \backslash G(F)} b\left(\gamma_{0}g\right) \rho\left(g\right)\one_{V(\OO)}(v)dg\\
= & \ \int b\left(\gamma_{0}\left(\left(\begin{smallmatrix} 1 & t\\ & 1 \end{smallmatrix} \right)\left(\begin{smallmatrix} a_1^{-1} & \\ & a_1 \end{smallmatrix} \right),\left(\begin{smallmatrix} 1 & t \\ & 1 \end{smallmatrix} \right)\left(\begin{smallmatrix} a_2^{-1} & \\ & a_2 \end{smallmatrix} \right),\left(\begin{smallmatrix} 1 & t \\ & 1 \end{smallmatrix} \right)\left(\begin{smallmatrix} a_3^{-1} & \\ & a_3 \end{smallmatrix} \right)\right)\right)\nonumber \\
\times & \ \rho\left(\left(\left(\begin{smallmatrix} 1 & t\\ & 1 \end{smallmatrix} \right)\left(\begin{smallmatrix} a_1^{-1} & \\ & a_1 \end{smallmatrix} \right),\left(\begin{smallmatrix} 1 & t \\ & 1 \end{smallmatrix} \right)\left(\begin{smallmatrix} a_2^{-1} & \\ & a_2 \end{smallmatrix} \right),\left(\begin{smallmatrix} 1 & t \\ & 1 \end{smallmatrix} \right)\left(\begin{smallmatrix} a_3^{-1} & \\ & a_3 \end{smallmatrix} \right)\right)\right)\one_{V(\OO)}(v)dt\prod_{i=1}^3|a_i|^2d^\times a_i\,,  \nonumber
\end{align}
where the integral is over $ F \times F^{\times 3}$.    Now
\begin{align*}
& \ \rho\left(\left(\left(\begin{smallmatrix} 1 & t\\ & 1 \end{smallmatrix} \right)\left(\begin{smallmatrix} a_1^{-1} & \\ & a_1 \end{smallmatrix} \right),\left(\begin{smallmatrix} 1 & t \\ & 1 \end{smallmatrix} \right)\left(\begin{smallmatrix} a_2^{-1} & \\ & a_2 \end{smallmatrix} \right),\left(\begin{smallmatrix} 1 & t \\ & 1 \end{smallmatrix} \right)\left(\begin{smallmatrix} a_3^{-1} & \\ & a_3 \end{smallmatrix} \right)\right)\right)\one_{V(\OO)}(v)\\
& \ = \psi(t\mathcal{Q}(v))\one_{V(\OO)}(a^{-1}v)\bar{\chi}_{\mathcal{Q}}(a)\prod_{i=1}^3|a_i|^{-d_i/2}\,.
\end{align*}
Thus \eqref{Ib} is equal to 
\begin{align} \label{before:val:lem}
\begin{split}
& \ \int b\left(\gamma_{0}\left(\left(\begin{smallmatrix} 1 & t\\ & 1 \end{smallmatrix} \right)\left(\begin{smallmatrix} a_1^{-1} & \\ & a_1 \end{smallmatrix} \right),\left(\begin{smallmatrix} 1 & t \\ & 1 \end{smallmatrix} \right)\left(\begin{smallmatrix} a_2^{-1} & \\ & a_2 \end{smallmatrix} \right),\left(\begin{smallmatrix} 1 & t \\ & 1 \end{smallmatrix} \right)\left(\begin{smallmatrix} a_3^{-1} & \\ & a_3 \end{smallmatrix} \right)\right)\right) \\
& \ \times \psi\left(t\mathcal{Q}(v)\right)\one_{V(\OO)}(a^{-1}v)\overline{\chi}_\mathcal{Q}(a)dt \prod_{i=1}^3|a_i|^{2-d_i/2}d^\times a_i\\
& \ = \int \sum_{k,j=0}^\infty q^{2j}
\one_{k+2j}\left(\gamma_{0}\left(\left(\begin{smallmatrix} 1 & t\\ & 1 \end{smallmatrix} \right)\left(\begin{smallmatrix} a_1^{-1} & \\ & a_1 \end{smallmatrix} \right),\left(\begin{smallmatrix} 1 & t \\ & 1 \end{smallmatrix} \right)\left(\begin{smallmatrix} a_2^{-1} & \\ & a_2 \end{smallmatrix} \right),\left(\begin{smallmatrix} 1 & t \\ & 1 \end{smallmatrix} \right)\left(\begin{smallmatrix} a_3^{-1} & \\ & a_3 \end{smallmatrix} \right)\right)\right) \\
& \ \times \psi\left(t\mathcal{Q}(v)\right)\one_{V(\OO)}(a^{-1}v)\overline{\chi}_\mathcal{Q}(a)dt \prod_{i=1}^3|a_i|^{2-d_i/2}d^\times a_i\,. \end{split}
\end{align}

Recall that 
$$
\gamma_0=\gamma_{0,0,0}=\begin{pmatrix} * \\ \begin{matrix} 
1 & 1 & 1 & 0 & 0 & 0\\ 0 & 0 & 0 &-1 & 1 &  0\\
0 & 0 & 0 & -1 & 0 & 1  \end{matrix}
\end{pmatrix},$$
hence
$$
\gamma_{0}\left(\left(\begin{smallmatrix} a_1^{-1} & a_1t\\ & a_1 \end{smallmatrix} \right),\left(\begin{smallmatrix} a_2^{-1} & a_2t \\ & a_2 \end{smallmatrix} \right),\left(\begin{smallmatrix} a_3^{-1} & a_3t\\ & a_3 \end{smallmatrix} \right)\right) = \begin{pmatrix}
* \\ \begin{matrix}
a_1^{-1} &a_2^{-1} &a_3^{-1} &ta_1&ta_2&ta_3\\
0&0&0&-a_1&a_2&0\\
0&0&0&-a_1&0&a_3\end{matrix}
\end{pmatrix}.
$$
Thus by Lemma \ref{lem:Iwasawa} we have that \eqref{before:val:lem} is equal to 
\begin{align} \label{adone}
&\sum_{k=0}^\infty\sum_{j=0}^\infty q^{2j}\int
 \psi\left(t\mathcal{Q}(v)\right)\one_{V(\OO)}(a^{-1}v)\overline{\chi}_\mathcal{Q}(a)\prod_{i=1}^3|a_i|^{2-d_i/2}d^\times a dt\,,
\end{align}
where the integral is over $a,t$ such that 
\begin{align*}
q^{-k-2j}=\max(|ta_1a_2a_3|,|a_1|,|a_2|,|a_3|,|a_1^{-1}a_2a_3|,|a_2^{-1}a_1a_3|,|a_3^{-1}a_1a_2|)\,.
\end{align*}
Note that
\begin{align*}
&\int_{|t| \leq q^{-k-2j}|a_1a_2a_3|^{-1}}
 \psi\left(t\mathcal{Q}(v)\right)dt=q^{-k-2j}|a_1a_2a_3|^{-1}\one_{\OO}\left(\frac{\mathcal{Q}(v)\varpi^{k+2j}}{a_1a_2a_3}\right).
\end{align*}
Using this fact we can simplify the $t$ integral in  \eqref{adone} to see that 
\begin{align*}
I(b,\one_{V(\OO)})(v)=& \ \sum_{k,j=0}^\infty q^{-k}\int\one_{\OO}\left(\frac{\mathcal{Q}(v)\varpi^{k+2j}}{a_1a_2a_3}\right)
\one_{V(\OO)}(a^{-1}v)\bar{\chi}_\mathcal{Q}(a)\prod_{i=1}^3|a_i|^{1-d_i/2}d^\times a \\
- & \ \sum_{k,j=0}^\infty q^{-k-1}\int\one_{\OO}\left(\frac{\mathcal{Q}(v)\varpi^{k+2j+1}}{a_1a_2a_3}\right)
\one_{V(\OO)}(a^{-1}v)\bar{\chi}_\mathcal{Q}(a)\prod_{i=1}^3|a_i|^{1-d_i/2}d^\times a\,,
\end{align*}
where the first integral is over the set of $a \in (F^\times)^3$ such that 
\begin{align*}
q^{-k-2j} \geq \max(|a_1|,|a_2|,|a_3|,|a_1^{-1}a_2a_3|,|a_2^{-1}a_1a_3|,|a_3^{-1}a_1a_2|)\,,
\end{align*}
and the second integral is over the set of $a \in (F^{\times})^3$ such that 
\begin{align*}
q^{-k-2j-1}\geq \max(|a_1|,|a_2|,|a_3|,|a_1^{-1}a_2a_3|,|a_2^{-1}a_1a_3|,|a_3^{-1}a_1a_2|)\,.
\end{align*}
If we then take a change of variables $(a_1,a_2,a_3) \mapsto \varpi^{k+2j}(a_1,a_2,a_3)$ to the first integral and $(a_1,a_2,a_3) \mapsto \varpi^{k+2j+1}(a_1,a_2,a_3)$ in the second integral we obtain the expression in the statement of the proposition.
\end{proof}

For the purpose of proving Proposition \ref{prop:1cell} we also require a bound on a related integral:

\begin{lem} \label{lem:abs:open}  Assume that $v \in V'(F)$. 
One has
\begin{align*}
    &\int_{N_0(F) \backslash G(F)}| b\left(\gamma_{0}g\right) \rho\left(g\right)\one_{V(\OO)}(v)|dg \\&\leq\begin{cases}  \prod_{i=1}^3(\mathrm{ord}(v_i)+1)^3|v_i|^{1-d_i/2} \one_{V(\OO)}(v) & \textrm{ if no }v_i=0\,,\\
\prod_{i=2}^3(\mathrm{ord}(v_i)+1)^4|v_i|^{2-d_i/2-d_1/2} \one_{V(\OO)}(v) &\textrm{ if }v_1=0\,.\end{cases}
\end{align*}
\end{lem}
Here in the lemma $\mathrm{ord}(v_i)$ is the minimum of the $v$-adic valuations of the entries of $v_i$.

\begin{proof}
Arguing as in the proof of Proposition \ref{prop:I0comp} we see that 
\begin{align} \label{adonebound}
\int_{N_0(F) \backslash G(F)}| b\left(\gamma_{0}g\right) \rho\left(g\right)\one_{V(\OO)}(v)|dg =\sum_{k=0}^\infty\sum_{j=0}^\infty q^{2j}\int
\one_{V(\OO)}(a^{-1}v)\prod_{i=1}^3|a_i|^{2-d_i/2}d^\times a dt\,,
\end{align}
where the integral is over $a,t$ such that 
\begin{align*}
q^{-k-2j} =\max(|ta_1a_2a_3|,|a_1|,|a_2|,|a_3|,|a_1^{-1}a_2a_3|,|a_2^{-1}a_1a_3|,|a_3^{-1}a_1a_2|)\,.
\end{align*}
Note that
\begin{align*}
\int_{|t| \leq q^{-k-2j}|a_1a_2a_3|^{-1}}dt
=q^{-k-2j}|a_1a_2a_3|^{-1}\,.
\end{align*}
Using this fact we can simplify the $t$ integral in \eqref{adonebound} to see that it is bounded by 
\begin{align*}
\sum_{k=0}^\infty\sum_{j=0}^\infty q^{-k}\int
\one_{V(\OO)}(a^{-1}v)\prod_{i=1}^3|a_i|^{1-d_i/2}d^\times a\,,
\end{align*}
where the integral is over $a \in (F^\times)^3$ such that 
$$
q^{-k-2j} \geq \max(|a_1|,|a_2|,|a_3|,|a_1^{-1}a_2a_3|,|a_2^{-1}a_1a_3|,|a_3^{-1}a_1a_2|)\,.
$$ 
We take a change of variables $(a_1,a_2,a_3) \mapsto \varpi^{k+2j}(a_1,a_2,a_3)$ to see that this is equal to
\begin{align} \label{almost:last}
\sum_{k=0}^\infty\sum_{j=0}^\infty q^{-k}\int
\one_{V(\OO)}((\varpi^{k+2j}a)^{-1}v)\prod_{i=1}^3\left(\frac{|a_i|}{q^{k+2j}}\right)^{1-d_i/2}d^\times a\,,
\end{align}
where the integral is over the $a_i$ such that $1 \geq \max(|a_1|,|a_2|,|a_3|,|a_1^{-1}a_2a_3|,|a_2^{-1}a_1a_3|,|a_3^{-1}a_1a_2|)$.
If all of the $v_i$ are nonzero then we note that \eqref{almost:last} is bounded by the analogous quantity where we take the integral to be over all $a \in \OO^3$, and it is easy to obtain the bound claimed in the lemma from this expression.  

Now assume $v_1=0$ (so $v_2 \neq 0 \neq v_3$).  In this case \eqref{almost:last} is bounded by 
\begin{align*}
&\sum_{k=0}^\infty\sum_{j=0}^\infty q^{-k}\int_{\OO^3}\one_{\OO}\left(\frac{a_2a_3}{a_1}\right)
\one_{V_2(\OO) \times V_3(\OO)}\left(\varpi^{-k-2j}\left(\frac{v_2}{a_2},\frac{v_3}{a_3}\right)\right)\prod_{i=1}^3\left(\frac{|a_i|}{q^{k+2j}}\right)^{1-d_i/2}d^\times a\,\\
& \leq \sum_{k=0}^\infty\sum_{j=0}^\infty q^{-k}\int_{\OO^2}
\one_{V_2(\OO) \times V_3(\OO)}\left(\varpi^{-k-2j}\left(\frac{v_2}{a_2},\frac{v_3}{a_3}\right)\right)\prod_{i=2}^3(\mathrm{ord}(a_i)+1)\left(\frac{|a_i|}{q^{k+2j}}\right)^{2-d_i/2-d_1/2}d^\times a\,.
\end{align*}
It is easy to obtain the lemma from this bound.
\end{proof}

\section{Bounds on integrals in the non-Archimedean case} \label{sec:bound:na}
In this section $F$ is a characteristic zero non-Archimedean local field and $K=\mathrm{Sp}_6(\OO)$. Fix
$$
(f_1,f_2) \in \mathcal{S}_{BK}(X(F),K) \times \mathcal{S}(V(F))\,.
$$
We bound the integrals attached to these functions that appeared in the proof of Theorem \ref{thm:main}.  These bounds will be used to deduce the absolute convergence statement of Proposition \ref{prop:1cell} below.
All implicit constants in this section are allowed to depend on $f_1$ and $f_2$.

\begin{prop}  \label{prop:open:bound:na} 
For $v \in V'(F)$ 
one has 
\begin{align*} 
\int_{N_0(F) \backslash G(F)} |f_1(\gamma_0g)\rho(g)f_2(v)|dg \ll \begin{cases}\prod_{i=1}^3|v_i|^{-1-d_i/2} & \textrm{ if all }v_i \neq 0\,,\\
|v_2|^{-d_2/2-d_1/2}|v_3|^{-d_3/2-d_1/2} & \textrm{ if }v_1 =0\,.
\end{cases}
\end{align*}
As a function of $v$ this integral has support in the intersection of a compact subset of $V(F)$ with $V'(F)$.  Thus $I(f_1,f_2)(v)$ admits the same bound and has support in a compact subset of $V(F)$. 
\end{prop}

\begin{proof}  
We decompose the Haar measure $dg$ as in the proof of Proposition \ref{prop:I0comp}.  Arguing as in that proposition we see that the integral in current proposition is equal to 
\begin{align} \label{na:int:to:bound2}
&\int_{(F^\times)^3 \times F \times K} \left|f_1\left(\gamma_{0}\left(\left(\begin{smallmatrix} 1 & t\\ & 1 \end{smallmatrix} \right)\left(\begin{smallmatrix} a_1^{-1} & \\ & a_1 \end{smallmatrix} \right),\left(\begin{smallmatrix} 1 & t \\ & 1 \end{smallmatrix} \right)\left(\begin{smallmatrix} a_2^{-1} & \\ & a_2 \end{smallmatrix} \right),\left(\begin{smallmatrix} 1 & t \\ & 1 \end{smallmatrix} \right)\left(\begin{smallmatrix} a_3^{-1} & \\ & a_3 \end{smallmatrix} \right)\right)k\right) \right|\\& \times |\rho(k)f_2(a^{-1}v)|\left(\prod_{i=1}^3|a_i|^{2-d_i/2}\right)d^\times a dt dk\,. \nonumber
\end{align}

Now
\begin{align} \label{mta0} \begin{split}
& \ \ \ \left|\gamma_{0}\left(\left(\begin{smallmatrix} 1 & t\\ & 1 \end{smallmatrix} \right)\left(\begin{smallmatrix} a_1^{-1} & \\ & a_1 \end{smallmatrix} \right),\left(\begin{smallmatrix} 1 & t \\ & 1 \end{smallmatrix} \right)\left(\begin{smallmatrix} a_2^{-1} & \\ & a_2 \end{smallmatrix} \right),\left(\begin{smallmatrix} 1 & t \\ & 1 \end{smallmatrix} \right)\left(\begin{smallmatrix} a_3^{-1} & \\ & a_3 \end{smallmatrix} \right)\right)k\right|\\
&= \max(|ta_1a_2a_3|,|a_1|,|a_2|,|a_3|,|a_1^{-1}a_2a_3|,|a_2^{-1}a_1a_3|,|a_3^{-1}a_1a_2|)\\
&=:  m(t,a)\,. \end{split}
\end{align}

By Lemma \ref{lem:bounded} this quantity is bounded for $a,t$ in the support of the integrand in \eqref{na:int:to:bound2}, and $f_1$ itself satisfies the bound
\begin{align} \label{f1:bound}
|f_1(g)| \ll |g|^{-2}\,.
\end{align}
For $a \in F^3$ let $|a|:=\max_i|a_i|$. Let
\begin{align} \label{k:int}
\widetilde{f}_2(v):=\int_K |\rho(k)f_2(v)|dk\,.
\end{align}

Assume for the moment that no $v_i$ is zero.
For some $c \in \RR_{>0}$ \eqref{na:int:to:bound2} is bounded by a constant times 
\begin{align} \label{na:int:to:bound3}
\int_{|a|\leq c} \int_{|t| \leq \frac{c}{|a_1a_2a_3|}}
m(t,a)^{-2} 
\widetilde{f}_2(a^{-1}v)\left(\prod_{i=1}^3|a_i|^{2-d_i/2}\right) dt d^\times a\,.  
\end{align}
For $|a| \leq c$ one has
\begin{align} \label{bound:for:later}
m(t,a) \geq |a_1| \geq c^{-2}|a_1a_2a_3|\,.
\end{align}
Thus \eqref{na:int:to:bound3} is bounded by 
\begin{align*}
&
c^4  \int_{|a|\leq c} \int_{|t| \leq \frac{c}{|a_1a_2a_3|}}
\widetilde{f}_2(a^{-1}v)\left(\prod_{i=1}^3|a_i|^{-d_i/2}\right)dt d^\times a \\
&\ll_c \int_{|a|\leq c} \widetilde{f}_2(a^{-1}v)\left(\prod_{i=1}^3|a_i|^{-1-d_i/2}\right)
d^\times a \,.
\end{align*} 
Since $f_2$ is a Schwartz function (in the usual sense) this has compact support as a function of $v \in V(F)$.  
Moreover, it is bounded by a constant times 
$\prod_{i=1}^3|v_i|^{-1-d_i/2}$.

Now assume that $v_1=0$, which implies both $v_2$ and $v_3$ are nonzero.  In this case rather than using the bound  \eqref{na:int:to:bound3} we use the stronger bound
\begin{align} \label{na:int:to:bound4}
\int_{\substack{|a|\leq c\\ |{a_2a_3}| \leq c|a_1|}} \int_{|t| \leq \frac{c}{|a_1a_2a_3|}}
m(t,a)^{-2} 
\widetilde{f}_2(a^{-1}v)\left(\prod_{i=1}^3|a_i|^{2-d_i/2}\right) dt d^\times a\,.  
\end{align}
This bound is still valid when none of the $v_i$ are zero, but we did not require it in that case.  We have 
$$
m(t,a) \geq |a_1|\,,
$$
so \eqref{na:int:to:bound4} is bounded by 
\begin{align*}
&\int_{\substack{|a|\leq c\\ |{a_2a_3}| \leq c|a_1|}} \int_{|t| \leq \frac{c}{|a_1a_2a_3|}}
|a_1|^{-d_1/2}
\widetilde{f}_2(a^{-1}v)\left(\prod_{i=2}^3|a_i|^{2-d_i/2}\right) dt d^\times a\\  &\ll\int_{\substack{|a|\leq c\\ |{a_2a_3}| \leq c|a_1|}} 
|a_1|^{-1-d_1/2}
\widetilde{f}_2(a^{-1}v)\left(\prod_{i=2}^3|a_i|^{1-d_i/2}\right) d^\times a\\
& \ll \int_{{|a_2|,|a_3|}\leq c}  \widetilde{f}_2(0,a_2^{-1}v_2,a_3^{-1}v_3)\left(\prod_{i=2}^3|a_i|^{-d_i/2-d_1/2}\right) d^\times a\,.
\end{align*}
{Note that the product is now over $2 \leq i \leq 3$ instead of $1 \leq i \leq 3$.}
It is clear that this {integral} is supported in a compact subset of $V_2(F) \times V_3(F)$ and that it is bounded by a constant times 
 $|v_2|^{-d_2/2-d_1/2}|v_3|^{-d_3/2-d_1/2}$.
\end{proof}

\section{Bounds on integrals in the Archimedean case}
\label{sec:bound:arch}

In this section $F$ is an Archimedean local field and $K \leq \mathrm{Sp}_6(F)$ is a maximal compact subgroup.  We estimate the local integrals defined in \S \ref{ssec:loc:func}.  The bounds obtained in this section will be used to prove Proposition \ref{prop:1cell}, the absolute convergence statement used in the proof of Theorem \ref{thm:main}.  As usual, the bound in the Archimedean case is slightly harder to prove than in the non-Archimedean case, but the basic outline of the proof is the same.
We fix
$$
(f_1,f_2) \in \mathcal{S}_{BK}(X(F),K) \times \mathcal{S}(V(F))\,.
$$
All implicit constants are allowed to depend on $f_1$, $f_2$.
 
 The following lemma will often be used below:
 \begin{lem} \label{lem:basic:bound2}
 Let $A,B \in \RR_{>0}$, $C \in \RR_{\geq 0}$ and let $x \in F^\times$.  If $A>B$ and $A \neq B+C$ one has
 \begin{align*}
 \int_{F^\times} \max(|a^{-1}x|,1)^{-A}|a|^{-B}\max(|a|,1)^{-C}d^\times a \ll_{A,B,C}\max(|x|,1)^{-\min(A,B+C)}\min(|x|,1)^{-B}\,.
 \end{align*}
 \end{lem}
 \begin{proof}
 We break the integral up into two ranges corresponding to $|a| \leq 1$ and $|a|>1$.  If $|x|<1$ then in the first range the integral is
\begin{align}
\int_{0<|x|<|a| \leq 1}|a|^{-B}d^\times a+\int_{|a| \leq |x|}|a|^{A-B}|x|^{-A }d^\times a \ll_{A,B} |x|^{-B}\,. 
\end{align}
 If $|x| \geq 1$ then in the first range the integral is 
 $$
 \int_{|a| \leq 1}|a|^{A-B}|x|^{-A} d^\times a \ll_{A,B}|x|^{-A}\,.
 $$
 Now consider the second range, in which $|a|>1$.  If $|x| \leq 1$ then this integral is 
 \begin{align*}
    \int_{|a|>1}|a|^{-B-C}d^\times a \ll_{B,C} 1\,.
 \end{align*}
 If $|x| >1$ then this integral is
 \begin{align*}
    \int_{|a|>1}\max(|a^{-1}x|,1)^{-A}|a|^{-B-C}d^\times a&=\int_{|x| <|a|}|a|^{-B-C}d^\times a+\int_{1<|a|\leq |x|}|a|^{A-B-C}|x|^{-A}d^\times a\\ &\ll_{A,B,C} |x|^{-B-C}+|x|^{-A}
    +|x|^{-B-C}\,.
 \end{align*}
 \end{proof}

\begin{prop}    \label{prop:open:bound:a}
  For any $N_1,N_2,N_3 \in \ZZ_{\geq 0}$ one has 
\begin{align*}
\int_{N_0(F) \backslash G(F)} |f_1(\gamma_0g)\rho(g)f_2(v)|dg \ll_{N_1,N_2,N_3} \prod_{i=1}^3 \max(|v_i|,1)^{-N_i}\min(|v_i|,1)^{-1-d_i/2}\,,
\end{align*}
for $v \in V'(F)$ with no $v_i=0$.  Thus $I(f_1,f_2)(v)$ admits the same bound.  
\end{prop}

\begin{proof}
Let $\widetilde{f}_2(v):=\int_{K}|\rho(k)f_2(v)|dk$ as before.  It is a continuous, rapidly decreasing function of $v$.
Let $N \in \ZZ_{\geq 0}$.  {Using Lemma \ref{lem:bounded} and arguing} as in the proof of Proposition \ref{prop:open:bound:na} we see that the integral in the current proposition is bounded by a constant depending on $N$ times
\begin{align*}
&\int_{(F^\times)^3 \times F} \min(m(t,a),1)^{-2}\max(m(t,a),1)^{-N}\widetilde{f}_2(a^{-1}v)\left(\prod_{i=1}^3|a_i|^{2-d_i/2}\right)dtd^\times a \,,
\end{align*}
with $m(t,a)$ defined as in \eqref{mta0}. 
For any $N_1,N_2,N_3 \in \ZZ_{\geq 0}$ this is bounded by a constant depending on $N_1,N_2,N_3$ times
\begin{align} \label{basic:bound0}
&\int_{(F^\times)^3 \times F} \min(m(t,a),1)^{-2}\max(m(t,a),1)^{-N}dt\prod_{i=1}^3\max(|a_i^{-1}v_i|,1)^{-N_i}|a_i|^{2-d_i/2}d^\times a \,.
\end{align}

For $a \in F^3$ let $|a|:=\max_i|a_i|$.
We separate the integral over $(F^\times)^3 \times F$ in \eqref{basic:bound0} into two ranges
\begin{align} \label{2ranges0}
\int_{\max(|a|,|ta_1a_2a_3|) <1} +\int_{\max(|a|,|ta_1a_2a_3|)\geq 1}\,.
\end{align}
We will bound the integral in each of these ranges separately.  All the implicit constants from this point on are allowed to depend on $N,N_1,N_2,N_3$.  We will always assume in the proof that $N_i>d_i/2+1$ because this will be necessary in our applications of 
Lemma \ref{lem:basic:bound2} below.  This is harmless because making the $N_i$ larger will only strengthen the bound asserted by the proposition.

In the first range in \eqref{2ranges0} we have
\begin{align}
m(t,a) \geq |a_1a_2a_3|
\end{align}
as in \eqref{bound:for:later}.  Thus  we see that this contribution is bounded by 
 a constant times
\begin{align}
&\int_{|a|\leq 1} \int_{|t| \leq \frac{1}{|a_1a_2a_3|}} |a_1a_2a_3|^{-2} \prod_{i=1}^3\max(|a_i^{-1}v_i|,1)^{-N_i}|a_i|^{2-d_i/2} dtd^\times a\,.
\end{align}
This in turn is bounded by a constant times
\begin{align} \label{forlater}
&\int_{|a|\leq 1} \prod_{i=1}^3\max(|a_i^{-1}v_i|,1)^{-N_i}|a_i|^{-1-d_i/2}d^\times a\,. 
\end{align}
For each factor we apply Lemma \ref{lem:basic:bound2} with $A=N_i$, $B=1+d_i/2$ and $C>N_i-1-d_i/2$ to see that this is 
\begin{align} \label{first:bound}
O\left(\prod_{i=1}^3\max(|v_i|,1)^{-N_i}\min(|v_i|,1)^{-1-d_i/2}\right).
\end{align}

In second range in  \eqref{2ranges0}
we have $m(t,a) \geq 1$.  {Note that for $n \in \ZZ_{>0}$ and $a_1,\dots,a_n \in \RR_{\geq 1}$ one has that
\begin{align} \label{triv}\max_{1\leq i \leq n}(a_i) \geq \left(\prod_{i=1}^na_i\right)^{1/n}\,.
\end{align}Thus
}
\begin{align} \label{fourth}
m(t,a) \geq \left(\max(|ta_1a_2a_3|,1)\max(|a_1|,1)\max(|a_2|,1)\max(|a_3|,1)\right)^{1/4}\,.
\end{align}
Thus the contribution of {the second range} is bounded by a constant depending on $N$ times 
\begin{align*}
    &\int_{(F^\times)^3 \times F}
    \max(|ta_1a_2a_3|,1)^{-N/4}dt\prod_{i=1}^3\max(|a_i^{-1}v_i|,1)^{-N_i}\max(|a_i|,1)^{-N/4}|a_i|^{2-d_i/2}d^\times a\\
    & \ll_N \int_{(F^\times)^3}
    \prod_{i=1}^3\max(|a_i^{-1}v_i|,1)^{-N_i}\max(|a_i|,1)^{-N/4}|a_i|^{1-d_i/2}d^\times a\,.
\end{align*}
{For $N$ large enough this is bounded by a constant times
\begin{align*}
    \int_{(F^\times)^3}
    \prod_{i=1}^3\max(|a_i^{-1}v_i|,1)^{-N_i}\max(|a_i|,1)^{-N/4}|a_i|^{-1-d_i/2}d^\times a\,.
\end{align*}
Choosing $N>4\max_{1 \leq i \leq 3}(N_i-d_i/2-1)$ and applying Lemma \ref{lem:basic:bound2} on the $i$th factor with $A=N_i$, $B=d_i/2+1$ and $C=N/4$  we arrive at a bound of
$$
O_{N_1,N_2,N_3}\left(\prod_{i=1}^3\max(|v_i|,1)^{-N_i}\min(|v_i|,1)^{-1-d_i/2} 
\right),
$$
which is the same as \eqref{first:bound}.}
\end{proof}

We also require the analogous bound when some $v_i$ is zero.  

\begin{prop}    \label{prop:open:bound:a2}
  For any $N_2,N_3 \in \ZZ_{\geq 0}$ one has 
\begin{align} \label{a:int:to:bound}
\int_{N_0(F) \backslash G(F)} |f_1(\gamma_0g)\rho(g)f_2(v)|dg \ll_{N_2,N_3} \prod_{i=2}^3 \max(|v_i|,1)^{-N_i}\min(|v_i|,1)^{-d_1/2-d_i/2}\,,
\end{align}
for $v \in V'(F)$ with $v_1=0$  Thus $I(f_1,f_2)(v)$ admits the same bound.  
\end{prop}

\begin{proof}
Arguing as in the proof of Proposition \ref{prop:open:bound:a} we see that for any $N,N_2,N_3 \in \ZZ_{\geq 0}$ this is bounded by a constant depending on $N,N_2,N_3$ times
\begin{align} \label{basic:bound}
&\int_{(F^\times)^3 \times F} \min(m(t,a),1)^{-2}\max(m(t,a),1)^{-N}dt|a_1|^{2-d_1/2}\prod_{i=2}^3\max(|a_i^{-1}v_i|,1)^{-N_i}|a_i|^{2-d_i/2}d^\times a 
\end{align}
with $m(t,a)$ defined as in \eqref{mta0}.
We begin by dividing the integral into 
ranges as follows:
\begin{align} \label{a1:divide}
\int_{\max(|a|,|a_1^{-1}a_2a_3|,|ta_1a_2a_3|)<1}+\int_{\max(|a|,|a_1^{-1}a_2a_3|,|ta_1a_2a_3|)\geq 1}\,.
\end{align}
To ease notation, all constants in this proof are allowed to depend on $N,N_2,N_3$ and the $d_i$.  We will also assume that $N_i>d_i/2+d_1/2$ in order to justify our applications of Lemma \ref{lem:basic:bound2}.  This is harmless for our purposes.

Consider the first range in \eqref{a1:divide}.  We have 
$m(t,a) \geq |a_1|$, so this contribution is bounded by 
\begin{align*}
    &\int_{\substack{(F^\times)^3 \times F\\\max(|a|,|a_1^{-1}a_2a_3|,|ta_1a_2a_3|)<1}} dt|a_1|^{-d_1/2}\prod_{i=2}^3\max(|a_i^{-1}v_i|,1)^{-N_i}|a_i|^{2-d_i/2}d^\times a\\
    & \ll \int_{\substack{(F^\times)^3 \\\max(|a|,|a_1^{-1}a_2a_3|) \leq 1}} |a_1|^{-1-d_1/2}\prod_{i=2}^3\max(|a_i^{-1}v_i|,1)^{-N_i}|a_i|^{1-d_i/2}d^\times a\\
    & \ll \int_{\substack{(F^\times)^2\\ \max(|a_2|,|a_3|)<1}} \prod_{i=2}^3\max(|a_i^{-1}v_i|,1)^{-N_i}|a_i|^{-d_i/2-d_1/2}d^\times a_i\,.
\end{align*}
Here we have trivially estimated the integrals over $t$ and $a_1$.  For any $N>0$ this is bounded by 
\begin{align} \label{for:later}
    \int_{\substack{(F^\times)^2}} \prod_{i=2}^3\max(|a_i^{-1}v_i|,1)^{-N_i}\max(|a_i|,1)^{-N}|a_i|^{-d_i/2-d_1/2}d^\times a_i\,.
\end{align}
Choosing $N>\max_{2 \leq i \leq 3}(N_i-d_i/2-d_1/2)$ and applying Lemma \ref{lem:basic:bound2} on the $i$th factor with $A=N_i$, $B=d_i/2+d_1/2$ and $C=N$ we see that this integral is 
\begin{align} \label{second:O}
O_{N_2,N_3}\left(\prod_{i=2}^3\max(|v_i|,1)^{-N_i} \min(|v_i|,1)^{-d_i/2-d_1/2}\right).
\end{align}

In the second range in \eqref{a1:divide} we have {$m(t,a) \geq 1$.  Using \eqref{triv} we deduce that}
\begin{align*}
m(t,a) \geq \left(\max(|ta_1a_2a_3|,1)\max(|a_1^{-1}a_2a_3|,1)\max(|a_1|,1)\max(|a_2|,1)\max(|a_3|,1)\right)^{1/5},
\end{align*}
and hence the contribution of the second range to \eqref{basic:bound} is bounded by 
\begin{align*} 
&\int_{(F^\times)^3 \times F} \left(\max(|ta_1a_2a_3|,1)\max(|a_1^{-1}a_2a_3|,1)\max(|a_1|,1)\right)^{-N/5}dt\\
& \times |a_1|^{2-d_1/2}\prod_{i=2}^3\max(|a_i^{-1}v_i|,1)^{-N_i}\max(|a_i|,1)^{-N/5}|a_i|^{2-d_i/2}d^\times a \\
& \ll \int_{(F^\times)^3} \left(\max(|a_1^{-1}a_2a_3|,1)\max(|a_1|,1)\right)^{-N/5}|a_1|^{1-d_1/2}\\& \times \prod_{i=2}^3\max(|a_i^{-1}v_i|,1)^{-N_i}\max(|a_i|,1)^{-N/5}|a_i|^{1-d_i/2}d^\times a\,.
\end{align*}

The contribution of $|a_1|\geq 1$ is bounded by a constant depending on $N$ times 
\begin{align}
    \int_{(F^\times)^2}  \prod_{i=2}^3\max(|a_i^{-1}v_i|,1)^{-N_i}\max(|a_i|,1)^{-N/5{+1}}|a_i|^{-d_i/2}d^\times a_i\,.
\end{align}
Choosing $N>5\max_{2 \leq i \leq 3}(N_i-d_i/2){+5}$ and applying Lemma \ref{lem:basic:bound2} on the $i$th factor with $A=N_i$, $B=d_i/2$ and $C=N/5{-1}$ we see that this integral is 
\begin{align} \label{third:O}
O\left(\prod_{i=2}^3\max(|v_i|,1)^{-N_i} \min(|v_i|,1)^{-d_i/2}\right).
\end{align}
Since we have dealt with the contribution of $|a_1| \geq 1$, we are left with bounding
\begin{align*}
 &\int_{\substack{(F^\times)^3\\ |a_1|<1}} \max(|a_1^{-1}a_2a_3|,1)^{-N/5}|a_1|^{1-d_1/2} \prod_{i=2}^3\max(|a_i^{-1}v_i|,1)^{-N_i}\max(|a_i|,1)^{-N/5}|a_i|^{1-d_i/2}d^\times a\,.
\end{align*}
We break this two ranges, namely $|a_2a_3|<|a_1|<1$ and $|a_1| \leq \min(1,|a_2a_3|) $.  The first range is {bounded by }
\begin{align*}
 &\int_{\substack{(F^\times)^3\\ |a_2a_3|<|a_1|<1}}|a_1|^{-d_1/2} \prod_{i=2}^3\max(|a_i^{-1}v_i|,1)^{-N_i}\max(|a_i|,1)^{-N/5}|a_i|^{1-d_i/2}d^\times a\\
 &\ll\int_{\substack{(F^\times)^2}}(1+|a_2a_3|^{-d_1/2}) \prod_{i=2}^3\max(|a_i^{-1}v_i|,1)^{-N_i}\max(|a_i|,1)^{-N/5}|a_i|^{1-d_i/2}d^\times a\,.
\end{align*}
This is dominated by \eqref{for:later} (with $N$ replaced by $N/5$) and hence bounded by \eqref{second:O} for $N$ large enough.
Assuming without loss that $N/5+1-d_1/2>0$ the second range is
\begin{align*}
    &\int_{\substack{(F^\times)^3\\ |a_1|\leq \min(1,|a_2a_3|)}}|a_2a_3|^{-N/5}|a_1|^{N/5+1-d_1/2} \prod_{i=2}^3\max(|a_i^{-1}v_i|,1)^{-N_i}\max(|a_i|,1)^{-N/5}|a_i|^{1-d_i/2}d^\times a\\
    &\ll \int_{\substack{(F^\times)^2\\ }}|a_2a_3|^{-N/5}\min(1,|a_2a_3|)^{N/5+1-d_1/2}\prod_{i=2}^3\max(|a_i^{-1}v_i|,1)^{-N_i}\max(|a_i|,1)^{-N/5}|a_i|^{1-d_i/2}d^\times a\\
    &\leq \int_{\substack{(F^\times)^2\\ }}\prod_{i=2}^3\max(|a_i^{-1}v_i|,1)^{-N_i}\max(|a_i|,1)^{-N/5}|a_i|^{2-d_i/2-d_1/2}d^\times a\,.
\end{align*}
This is dominated by \eqref{for:later} (with $N$ replaced by $N/5$) and hence bounded by \eqref{second:O} for $N$ large enough.  
\end{proof}

\section{Absolute convergence} \label{sec:AC}

In this section we prove the absolute convergence statement
that makes the proof of the summation formula in \S \ref{sec:summation} rigorous.
We begin with the following lemma:
\begin{lem} \label{lem:prod:rule}
For $x \in F_\infty^n$ and $v|\infty$ let $|x|_v:=\max\{|x_i|_v:1 \leq i \leq n\}$.  Let $A>0$, $N>0$, $\beta \in \OO \cap F^\times$ be given.  If $\alpha \in \beta^{-1}\OO^n - 0$ then 
\begin{align*}
\prod_{v|\infty}\left(\max(|\alpha|_v,1)^{-N-A}\min(|\alpha|_v,1)^{-A}\right)\ll_{A,\beta} \prod_{v|\infty}\max(|\alpha|_v,1)^{-N}\,.
\end{align*}
\end{lem}

\begin{proof}
One has 
\begin{align*}
\prod_{v|\infty}\left(\max(|\alpha|_v,1)^{-N-A}\min(|\alpha|_v,1)^{-A}\right)=&\ \prod_{v|\infty}\left(\max(|\alpha|_v,1)^{-N}|\alpha|_v^{-A}\right)\\
=&\ |\alpha|_\infty^{-A}\prod_{v|\infty}\left(\max(|\alpha|_v,1)^{-N}\right).
\end{align*}
\end{proof}

For the remainder of the section we fix 
\begin{align}
(f_1,f_2) \in \mathcal{S}_{BK}(X(\A_F),K) \times \mathcal{S}(V(\A_F))\,,
\end{align}
where $K^\infty$ is an $\mathrm{Sp}_6(\A_F^\infty)$-conjugate of $\mathrm{Sp}_6(\widehat{\OO})$.  All implicit constants are allowed to depend on $f_1,f_2$.

\begin{prop} \label{prop:1cell}
The sum 
\begin{align}
\sum_{\xi \in V^{\mathrm{sm}}(F)} \int_{N_0(\A_F) \backslash G(\A_F)} |f_1(\gamma_0 g)\rho(g)f_2(\xi)| dg
\end{align}
converges.
\end{prop}

\begin{proof}
Let $S$ be a finite set of places of $F$ including the infinite places such that {$\psi$ is unramified outside of $S$, $\rho(k)\one_{V(\widehat{\OO}^S)}=\one_{V(\widehat{\OO}^S)}$ for $k \in \SL_2(\widehat{\OO}^S)$}, $f_1^S=b^S$ and $f_2^S=\one_{V(\widehat{\OO}^S)}$.  Let $V'' \subset V$ be the open subscheme of points $(\xi_1,\xi_2,\xi_3)$ such that no $\xi_i=0$.  
Let $\varepsilon>0$.  Using Lemma \ref{lem:abs:open} and Propositions \ref{prop:open:bound:na} and \ref{prop:open:bound:a} we have a bound on the sum {over $\xi \in V''(F)$} of a constant depending on $\varepsilon$ times
\begin{align*}
\sum_{\xi \in \beta^{-1}V(\OO) \cap V''(F)} \prod_{i=1}^3\left( \prod_{v |\infty} 
\max(|\xi_i|_v,1)^{-N_i}\min(|\xi_i|_v,1)^{-1-d_i/2}
\prod_{v \in S -\infty} |\xi_i|_v^{{-}1-d_i/2}\prod_{v \not \in S}(|\xi_i|_v^{1-d_i/2+\varepsilon})\right),
\end{align*}
for some $\beta \in F^\times \cap \OO$ divisible only by places in $S$.  Using Lemma \ref{lem:prod:rule} we see that this sum is dominated by a constant depending on $N \geq 0$ times
\begin{align*}
\sum_{\xi \in \beta^{-1}V(\OO) \cap V''(F)} \prod_{i=1}^3 \left(\prod_{v |\infty} 
\max(|\xi_i|_v,1)^{-N}
\prod_{v \in S -\infty} |\xi_i|_v^{{-}1-d_i/2}\prod_{v \not \in S}(|\xi_i|_v^{1-d_i/2+\varepsilon})\right).
\end{align*}
This is finite for $N$ large enough.

We still must bound the contribution of $\xi \in V^{\mathrm{sm}}(F)-V''(F)$.  By symmetry, it suffices to consider the contribution of $\xi \in V(F)$ such that $\xi_1=0$ and $\xi_2$ and $\xi_3$ are nonzero.  This contribution can be bounded using Lemma \ref{lem:abs:open}, Propositions \ref{prop:open:bound:na} and \ref{prop:open:bound:a2} and the argument above. 
\end{proof}

\section{A vanishing statement} \label{sec:appendix}

As a public service we state and prove the following theorem in greater generality than we need for the current paper.  Let $v_0$ be a place of $F$.  Let 
\begin{equation}\label{Xn}
   X_n:=[P_n,P_n] \backslash \mathrm{Sp}_{2n}\,, 
\end{equation}
where $P_n \leq \mathrm{Sp}_{2n}$ is the Siegel parabolic of \cite{Getz:Liu:BK}, and let
$$
\mathcal{S}_{BK}(X_n(F_{v_0}),K_{nv_0})
$$
be the Schwartz space of \cite{Getz:Liu:BK}, where $K_{nv_0} \leq \mathrm{Sp}_{2n}(F_{v_0})$ is a maximal compact subgroup that is conjugate to $\mathrm{Sp}_{2n}(\OO_{v_0})$ in the non-Archimedean case.  In loc.~cit.~global analogues 
$$
\mathcal{S}_{BK}(X_n(\A_F),K_n)
$$ 
were also defined.  
We again have a Fourier transform 
$$
\mathcal{F}:\mathcal{S}_{BK}(X_n(F_{v_0}),K_{nv_0}) \lto \mathcal{S}_{BK}(X_n(F_{v_0}),K_{nv_0})
$$ and a global analogue. When $n=3$ all of these reduce to the setting of the current paper.

Let $C_c^\infty(X_n(F_{v_0}),K_{nv_0})$ be the space of compactly supported smooth $K_{nv_0}$-finite functions on $X_n(F_{v_0})$.  It is equal to $C_c^\infty(X_n(F_{v_0}))$ if $v_0$ is non-Archimedean.  By \cite[Proposition 4.7]{Getz:Liu:BK} one has 
\begin{align}
    C_c^\infty(X_n(F_{v_0}),K_{nv_0}) <\mathcal{S}_{BK}(X_n(F_{v_0}),K_{nv_0})\,.
\end{align}

Let $M_n \leq P_n$ be the Levi subgroup of block diagonal matrices and let
\begin{align*}
    \omega:M_n(R) &\lto R^\times\\
    \left(\begin{smallmatrix} A & \\ & {}^tA^{-1} \end{smallmatrix} \right) &\longmapsto \det A\,.
\end{align*}
For $f \in \mathcal{S}_{BK}(X_n(\A_F),K_n)$, characters $\chi:F^\times \backslash \A_F^\times \to \CC^\times$ and $s \in \CC$ set
\begin{align}
    f_{\chi_s}(x):=\int_{M_n^{\mathrm{ab}}(\A_F)}\delta_P(m)^{1/2}\chi_s(\omega(m))f(m^{-1}x)dm\,.
\end{align}
This is a section of the induced representation $I(\chi_s)$ in the notation of \cite{Getz:Liu:BK}.  We will also use the obvious local analogue of this notation.  In the proofs in the rest of this section we will require the usual intertwining operator
$$
M_{w_0}:I(\chi_s) \lto I(\bar{\chi}_{-s})
$$
and the normalized version $M_{w_0}^*$ that appears in the work of Piatetski-Shapiro and Rallis and Ikeda.  We will use this in both local and global contexts.  We refer to \cite[\S 3]{Getz:Liu:BK} for notation.  
We let $E(g,f_{\chi_s})$ be the usual degenerate Siegel Eisenstein series \cite[(1.3.1)]{Getz:Liu:BK} on $\mathrm{Sp}_{2n}(\A_F)$.

\begin{thm}[Kudla-Rallis] \label{thm:A:BK}
If $f_{v_0} \in C_c^\infty(X_n(F_{v_0}),K_{nv_0}) <\mathcal{S}(X_n(F_{v_0}),K_{nv_0})$ for some non-Archimedean place $v_0$ of $F$, then for any $f^{v_0} \in \mathcal{S}_{BK}(X_n(\A_F^{v_0}),K_n^{v_0})$ 
$$
\mathrm{Res}_{s=\frac{n+1}{2}-m}E(g,\mathcal{F}(f^{v_0}f_{v_0})_{1_s})=0
$$
for integers $0 \leq m < \frac{n+1}{2}$, 
and 
$$
\mathrm{Res}_{s=\frac{n-1}{2}-m}E(g,\mathcal{F}(f^{v_0}f_{v_0})_{\chi_s})=0
$$
for quadratic characters $\chi$ and integers $0 \leq m < \frac{n-1}{2}$.
In particular, when $n=3$ the function $\mathcal{F}(f^{v_0}f_{v_0})$ satisfies assumption \eqref{A:BK}.
\end{thm}

\begin{proof}  
This is a refinement of \cite[Theorem 4.12]{Kudla:Rallis:firstterm}.  Unfortunately Kudla and Rallis have different assumptions regarding sections, so we explain how to deduce the theorem {using the argument of } loc.~cit.  We also warn the reader that in \cite{Kudla:Rallis:firstterm} the Kudla and Rallis assume that the number field in question is totally real.  However this is not used in the results we will quote below.

Let $s_0 \in \left\{\tfrac{n+1}{2}-m: m \in \ZZ, 0 \leq m < \tfrac{n+1}{2} \right\}$.  Then one has a $\mathrm{Sp}_{2n}(\A_F)$-intertwining map
\begin{align*}
A_{-1}:I(\chi_s)  &\lto \mathcal{A}(\mathrm{Sp}_{2n})\\
\Phi(s)&\longmapsto \mathrm{Res}_{s=s_0}E(g,\Phi(s))\,,
\end{align*}
where $\mathcal{A}(\mathrm{Sp}_{2n})$ is the space of automorphic forms on $\mathrm{Sp}_{2n}(\A_F)$.  If $\Phi(s)=\Phi_{v_0}(s)\Phi^{v_0}(s)$ is a standard section such that $\Phi_{v_0}(s)$ is in the space denoted by 
$$
R_n(V_1)\cap R_n(V_2) \leq I(\chi_{v_0s})
$$ 
in \cite[Proposition 4.2]{Kudla:Rallis:firstterm} then $A_{-1}(\Phi(s))=0$ by loc.~cit.  Here a standard section is a section whose restriction to a given maximal compact subgroup of $\mathrm{Sp}_{2n}(\A_F)$ is independent of $s$.  

For $f \in \mathcal{S}_{BK}(X(\A_F),K)$ the section $\mathcal{F}(f)_{\chi_s}$ is not standard, but since $d(s,\chi):=\prod_{v}d(s,\chi_v)$ is absolutely convergent for $\mathrm{Re}(s)>0$ 
the section $\mathcal{F}(f)_{\chi_s}$ is holomorphic at $s_0$ (see \cite[\S 3]{Getz:Liu:BK}).
Since the values of standard sections at $s_0$ span the space of $K$-finite vectors in $I(\chi_{s_0})$ as a vector space we deduce that 
$$
\mathrm{Res}_{s=s_0}E(g,\mathcal{F}(f)_{\chi_s})=0
$$
if $\mathcal{F}(f)_{\chi_{v_0s_0}} \in R_n(V_1) \cap R(V_2)$.  

We claim that
$\mathcal{F}(f)_{{v_0}\chi_s} \in R_n(V_1) \cap R(V_2)$
whenever $f \in C_c^\infty(X(F_{v_0}))$.  Proving the claim will complete the proof of the theorem.  
Now 
$$
\mathcal{F}(f)_{\chi_{v_0s}}=M_{w_0}^*f_{\bar{\chi}_{v_0-s}}
$$
by \cite[Theorem 4.4]{Getz:Liu:BK}.
In the notation of \cite[Proposition 6.5]{Kudla:Rallis:Degenerate} one has
$$
M_{n}^*(s)=\frac{1}{L(s-\tfrac{n-1}{2},\chi_{v_0})\prod_{r=1}^{\lfloor n/2 \rfloor}L(2s-n{+}2r,\chi^2_{v_0})}M_{w_0}\,.
$$
In particular $M_{w_0}^*$ (acting on sections in $I(\bar{\chi}_{v_0-s})$) is a nonvanishing entire function times $L(1+s+\tfrac{n-1}{2},\chi_{v_0})\prod_{i=1}^{\lfloor n/2 \rfloor}L(1+2s+n-2r,\chi^2_{v_0})M_n^*(-s)$, and thus $M_{w_0}^*f_{\bar{\chi}_{v_0-s}}$ is equal to $M_{n}^*f_{\bar{\chi}_{v_0-s}}$ up to a function that is holomorphic in $\mathrm{Re}(s)>0$.  
{Thus we can deduce our claim from \cite[\S 6]{Kudla:Rallis:Degenerate} as in the proof of \cite[Theorem 4.12]{Kudla:Rallis:firstterm}}.
\end{proof}

\section*{List of symbols}

\begin{center}
\begin{tabular}{l c r}
$b$& basic function &\eqref{b}\\
$c(x)$& cocharacter of $M$ &  \eqref{c:def} \\
$\chi_{\mathcal{Q}}$ & quadratic character &\eqref{chiQ}\\
$E(g;f_{\chi_s})$ & Eisenstein series & \eqref{Eisensteinseries}\\
$\mathcal{F}=\mathcal{F}_{BK,\psi}$ & Fourier transform on a BK space &\eqref{mathcalF}\\
$f_{\chi_s}$ &  local (global) Mellin transform  &  \eqref{fchis:local} (\eqref{fchis:global})\\
$G$ & $\SL_2^3$ & \eqref{Gembed}\\
$\gamma_i$ & representatives for $X(F)/G(F)$ & \eqref{gammas}\\
$G_{\gamma_i}$ & stabilizer of $\gamma_i$ & \eqref{Ggammas}\\
$\lvert g \rvert$& norm of Pl\"ucker embedding & \eqref{norm:def}\\
$H$ & similitude group & \eqref{H}\\
$I(f_1,f_2)$ & integral  & \eqref{Is:global}\\
$J_i$ & matrix corresponding to $\mathcal{Q}_i$ &\eqref{Ji}\\
$L(h)$& left translation action & \eqref{L(h)}\\
$\Lambda(h)$& element of $H(F)$ &\eqref{Lambdah}\\
$M$& Levi subgroup of  $P$ & \S \ref{ssec:simil}\\
$M^{ab}=[M,M] \backslash M$ &  abelianization of $M$ & \S \ref{ssec:BK}\\
$N$ & unipotent radical of $P$ & \S \ref{ssec:simil}\\
$N_0$ & stabilizer of $\gamma_0$ & \eqref{T0N0}\\
$\one_c$ & characteristic function of $\one_{[P,P](F)c(\varpi)\mathrm{Sp}_6(\OO)}$ & \eqref{onec}\\
$\omega$& character of $M$ & \eqref{omega}\\
$P$ &  Siegel parabolic & \eqref{Siegel}\\
$\mathrm{Pl}(g)$ & Pl\"ucker embedding & \eqref{Pl}\\
$\mathcal{Q}$& quadratic form on $V$ & \eqref{Q}\\
$\mathcal{Q}_i$ & quadratic form on $V_i$ &\eqref{Qi}\\
$\rho=\rho_{\psi}$ &  local (global) Weil representation &
\eqref{rholocal} (\eqref{rhoglobal})\\
$\mathcal{S}_{BK}(X(\A_F),K)$ &  BK Schwartz space $X(\A_F)$ &  \S \ref{sec:Schwartz}\\
$\mathcal{S}(V(F_v))$, $\mathcal{S}(V(\A_F))$,& usual Schwartz spaces & \S \ref{sec:intro}\\
$T$ & maximal torus of $G$ &  \eqref{T}\\
$T_0$ & subtorus of $T$ & \eqref{T0N0}\\
$\Theta_f$ & Theta function & \eqref{Theta}\\
$V_i$ & quadratic space of even dimension & \S \ref{sec:intro}\\
$V$&  $\prod_{i=1}^3V_i$ & \eqref{V}\\
$V'$ & open subscheme of $V$ & \eqref{Vprime} \\
$X$ & Braverman-Kazhdan space & \eqref{BKspaceX}\\
$Y$ & $\{ v \in V(R): \mathcal{Q}_1(v_1)=\mathcal{Q}_2(v_2)=\mathcal{Q}_3(v_3)=0\}$ &  \eqref{Y}\\
$Y^{\mathrm{sm}}$ & smooth locus in $Y$ &  \eqref{Ysm}
\end{tabular}
\end{center}

%----------------------------------------------------------------

\bibliography{refs}
\bibliographystyle{alpha}

\end{document}